\documentclass[a4paper]{amsart}
\usepackage{amsmath,amssymb,stmaryrd,amsthm,amscd,mathrsfs}
\usepackage[all,2cell]{xy}
\UseAllTwocells
\usepackage{hyperref}
\usepackage{leftidx}
\usepackage[usenames,dvipsnames,svgnames,table]{xcolor}
\usepackage{amsrefs}
\usepackage{marginnote}
\usepackage[british]{babel}
\usepackage{tikz-cd}
\usepackage[normalem]{ulem}

\newcommand{\A}{\mathcal{A}}
\newcommand{\B}{\mathcal{B}}
\newcommand{\C}{\mathcal{C}}
\newcommand{\D}{\mathcal{D}}
\newcommand{\E}{\mathcal{E}}
\newcommand{\F}{\mathcal{F}}
\newcommand{\G}{\mathcal{G}}
\newcommand{\I}{\mathcal{I}}
\newcommand{\clL}{\mathcal{L}}
\newcommand{\M}{\mathcal{M}}

\newcommand{\R}{\mathcal{R}}

\newcommand{\X}{\mathcal{X}}
\newcommand{\Y}{\mathcal{Y}}
\newcommand{\W}{\mathcal{W}}

\DeclareMathOperator*{\Mod}{\mathsf{Mod}-\!}
\DeclareMathOperator{\Hom}{\mathsf{Hom}}
\DeclareMathOperator{\RHom}{\mathbf{R}Hom}

\DeclareMathOperator{\Ext}{\mathsf{Ext}}
\DeclareMathOperator{\Tor}{\mathsf{Tor}}
\DeclareMathOperator{\Ker}{\mathsf{Ker}}

\DeclareMathOperator{\Coker}{\mathsf{Coker}}
\DeclareMathOperator{\coker}{coker}

\newcommand{\lifts}{\,\square\,}

\newcommand{\rightperp}[1]{#1^{\perp}}
\newcommand{\leftperp}[1]{{}^\perp #1}

\newcommand{\pd}{\mathrm{pd}}

\newcommand{\Proj}[1]{\hbox{\rm Proj}{#1}}

\newcommand{\Inj}[1]{\mathrm{Inj}\,{#1}}

\newcommand{\Rep}{\mathrm{Rep}}

\newcommand{\GProj}{\operatorname{\mathsf{GProj}}\nolimits}
 
 \newcommand{\GInj}{\operatorname{\mathsf{GInj}}\nolimits}

\newcommand{\Filt}{\mathrm{Filt}}

\newcommand{\class}{\mathcal}

\newcommand{\Ch}{\mathsf{C}}
\newcommand{\Chac}{\mathsf{C}_\mathsf{ac}}

\newcommand{\Ho}{\mathsf{Ho}}
\newcommand{\Der}{\mathsf{D}}

\newtheorem{theorem}{Theorem}[section]

\newtheorem{lemma}[theorem]{Lemma}
\newtheorem{setup}[theorem]{Setup}
\newtheorem{proposition}[theorem]{Proposition}
\newtheorem{corollary}[theorem]{Corollary}

\theoremstyle{definition}
\newtheorem{definition}[theorem]{Definition}
\newtheorem{fact}[theorem]{Fact}
\newtheorem{remark}[theorem]{Remark}

\newtheorem{example}[theorem]{Example}

\newtheorem*{problem}{Problem}
\newtheorem{ipg}[theorem]{}

\begin{document}

\title[Lifting recollements and model structures]{Lifting recollements of abelian categories \\ and model structures}


\author{Georgios Dalezios}
\author{Chrysostomos Psaroudakis}
\address{Institute of Algebra and Number Theory, University of Stuttgart, Pfaffenwaldring
57, 70569 Stuttgart, Germany}
\email{gdalezios@math.uoa.gr}
\address{Department of Mathematics, Aristotle University of Thessaloniki, Thessaloniki 54124, Greece}
\email{chpsaroud@math.auth.gr}

\begin{abstract}
We use Quillen model structures to show a systematic method to lift recollements of hereditary abelian model categories to recollements of their associated homotopy categories. To that end, we use the notion of Quillen adjoint triples and we investigate transfers of abelian model structures along adjoint pairs. Applications include liftings of recollements of module categories to their derived counterpart, liftings to homotopy categories that provide models for stable categories of Gorenstein projective and injective modules and liftings to homotopy categories of $n$-morphism categories over Iwanaga-Gorenstein rings.
\end{abstract}

\maketitle

\setcounter{tocdepth}{1}
\tableofcontents

\section{Introduction and main results}
\label{sec:intro}
Derived categories were introduced by Grothendieck and Verdier and play a substantial role in several branches of mathematics. One major aspect of derived categories is that several homological invariants can be formulated naturally in this setup. Classifying algebraic or geometric objects up to derived equivalence is certainly another important aspect of derived categories, 
since it provides canonical isomorphisms between the invariants \cite{Ktheoryinvariants}. 
One of the key ingredients of those developments is the triangulated structure of derived categories \cite{Verdier}. On the other hand, Beilinson-Bernstein-Deligne \cite{BBD} introduced recollements of triangulated categories formalizing Grothendieck's six functors for derived categories of sheaves. A recollement of triangulated categories is a short exact sequence of triangulated categories where both the quotient functor as well as the inclusion functor admit a left and a right adjoint. Recollement situations provide a very convenient framework for studying homological invariants. Of particular interest are recollements of derived categories that arise from recollements of abelian categories. 

Recollements of abelian categories had implicitly appeared in \cite{BBD} via the glueing method of t-structures along a recollement of triangulated categories. The associated hearts, which are abelian categories, give rise to a recollement situation. In the context of representation theory and motivated by highest weight categories and quasi-hereditary algebras, Cline, Parshall and Scott investigated first when a recollement of module categories lifts to a recollement of (bounded) derived categories \cite{CPS, CPS_2, CPS_3}. The language of recollements of module categories was not used in the latter papers, but that was exactly the context. Typically, recollements of module categories arise from pairs $(R,e)$, where $R$ is a ring and $e$ is an idempotent element in $R$. More precisely, any pair $(R,e)$ induces a recollement between the module categories  over the rings $R$, $R/ReR$ and $eRe$. 
Cline, Parshall and Scott showed in \textit{loc.cit} that a recollement of module categories with $ReR$ being stratifying induces a recollement at the level of derived module categories. In that spirit, Krause  characterized highest weight categories via a sequence of recollements of abelian categories which lift to derived categories \cite{Krause}. Going back to the homological invariants, Angeleri H\"{u}gel, Koenig, Liu and Yang in 
\cite{AKLY_1, AKLY_2, AKLY_3, AKLY_4} investigated  which invariants can be computed inductively along a sequence of recollements of derived module categories, when recollements lift or restrict to various levels of derived categories and when a derived version of the Jordan-H\"older theorem holds. Several key ingredients of the latter program as well as examples, are based on the mechanism of lifting recollements of module categories to the associated derived categories.

Derived categories of nice enough abelian categories are instances of Quillen homotopy categories. Indeed, there exist certain well-known projective and injective model category structures on unbounded chain complexes having as (common) homotopy category the derived category. These model structures share the property of being \textit{hereditary} and \textit{abelian}, in short, this means that they are in bijection with certain hereditary and complete cotorsion pairs (Ext-orthogonal classes) in the underlying abelian category. The theory of abelian model structures has been developed extensively by Hovey and Gillespie in a series of papers such as \cite{Gil2011, hovey2002}. The upshot of working with a hereditary abelian model structure is that its homotopy category is canonically triangulated, in fact, it coincides with the stable category of a Frobenius category \cite{hoveybook, Gil2011, Happelbook}. Under this homotopical point of view and motivated by the discussion so far, we study the following general question. 

\begin{problem}
When a recollement of hereditary abelian model categories can be lifted to a recollement of their associated triangulated homotopy categories?
\end{problem}

In order to approach this problem, we first need to deal with the lifting of certain adjoint triples. Let us first look at the simple situation where we are given an associative ring $R$ with and idempotent $e$ and an adjoint triple of categories of chain complexes,
\begin{equation}
\label{eq:intro_adj_tri}
  \xymatrix@C=4pc{
\Ch(R) \ar@<0.0ex>[r] ^-{e}   &
\Ch(eRe) \ar@/^1.4pc/[l] ^-{r:=\mathrm{Hom}_{eRe}(eR,-)}, \ar@/_1.4pc/[l] _-{l:=Re\otimes_{eRe}-}
 }
\end{equation}
where the middle exact functor is multiplication by $e$ (degreewise). 
It is easy to see that since $e$ is exact, this adjoint triple lifts to one at the level of derived categories 
\begin{equation}
\label{eq:intro_adj_tri_der}
  \xymatrix@C=4pc{
 \Der(R)    \ar[r] |-{\, \textbf{R}e\cong\textbf{L}e \,}&
\Der(eRe). \ar@/^1.4pc/[l] ^-{\textbf{R}r} \ar@/_1.4pc/[l] _-{\textbf{L}l}  
  }
\end{equation} 
In terms of model categories, the upper half of the adjoint triple in (\ref{eq:intro_adj_tri}), i.e. the adjoint pair $(l,e)$, is well-behaved with respect to the (usual) projective model structure on chain complexes (the left derived functor $\mathbf{L}l$ is computed via resolutions by semi-projective complexes, which are the cofibrant objects in the projective model structure), while the lower part, i.e.\ the adjoint pair $(e,r)$ is well-behaved with respect to the (usual) injective model structure. To be more precise, by well-behaved we mean the formation of \textit{Quillen adjoint pairs}, which are adjoint pairs between model structures that induce  adjoint pairs at the level of homotopy categories. In fact, the projective and the injective model structures in this example are interconnected, in the sense that we have a commutative square\footnote{Here $\Ch(R)_{\mathsf{proj}}$, resp., $\Ch(R)_{\mathsf{inj}}$ just denotes the abelian category $\Ch(R)$, viewed as a model category with the projective, resp., injective model structure. Hence commutativity of the diagram, strictly speaking, expresses the identity natural transformation from $e\colon \Ch(R)\rightarrow\Ch(R)$ to itself.}

\begin{equation*}
  \xymatrix@C=3pc{
 \Ch(R)_{\mathsf{proj}}  \ar[r]^-{e}  \ar[d]_-{\mathrm{id}} &  \Ch(eRe)_{\mathsf{proj}} \ar[d]^-{\mathrm{id}} \\
 \Ch(R)_{\mathsf{inj}}  \ar[r]^-{e} & \Ch(eRe)_{\mathsf{inj}}
  }\!
\end{equation*}
where the top functor is right Quillen between projective model structures on chain complexes, the bottom functor is left Quillen between injective model structures and the vertical (left Quillen) identity functors induce equivalences at the level of derived categories. This is all the information needed to obtain diagram (\ref{eq:intro_adj_tri_der}).

We formalize the above situation by introducing the concept of \textit{Quillen adjoint triples} (Definition \ref{def:Quillen_adjoint_triple}). These are adjoint triples between model structures that behave well with respect to lifting to homotopy categories. Our definition involves certain natural transformations that compare composites of left and right Quillen functors and is inspired by work of Shulman \cite{Shulman}, who used the concept of \textit{double categories} in order to study such composites. The precise relation between  Shulman's work and our Quillen adjoint triples is given in Section \ref{sec:a_double_psadofunctor}.  

We then focus on recollements of Grothendieck categories 
\begin{equation*}
  \xymatrix@C=4pc{
\A   \ar@<0.0ex>[r] |-{\, i\,}   &  
  \B \ar@/^1.4pc/[l]^-{p}    \ar@/_1.4pc/[l] _-{q}  \ar[r] |-{\, e \,}&
  \C \ar@/_1.4pc/[l]  _-{l}  \ar@/^1.4pc/[l]^-{r}
  }\!
\tag*{$\mathsf{R_{ab}}(\A,\B,\C)$}
\end{equation*}
such that the following setup is satisfied:
\begin{itemize}
\item $\B$ admits two hereditary abelian model structures with Hovey triples $\B_{\mathsf{proj}}:=(\C_{\B},\W_{\B},\B)$ and $\B_{\mathsf{inj}}:=(\B,\W_{\B},\F_{\B})$ which are Quillen equivalent via the identity functor. Also, we assume same homotopical structure on $\C$.

\item The functor $e\colon \B\rightarrow\C$ preserves trivial objects, i.e., $e(\W_{\B})\subseteq\W_{\C}$.
\end{itemize}
The class $\W_{\B}$ is commonly the class of trivial objects for $\B_{\mathsf{proj}}$ and $\B_{\mathsf{inj}}$, thus they share the same homotopy category, which we denote by $\Ho(\B)_{\mathsf{proj/inj}}$. The same fact applies to $\C$, and we denote by $\Ho(\C)_{\mathsf{proj/inj}}$ the common homotopy category of $\C_{\mathsf{proj}}$ and $\C_{\mathsf{inj}}$.

We remark that this setup is typical and not too strong; we do not assume for instance that $\mathcal{A}$ admits any model structure a priori. We rather study, in Section \ref{sec:transfer}, \textit{constructions} of model structures on $\mathcal{A}$ along the functor $i$. By definition, we say that the \textit{right-lifted} abelian model structure, along the right adjoint $i$, exists on $\mathcal{A}$, if there is an (abelian) model structure on $\mathcal{A}$ with trivial objects $i^{-1}(\W_{\B})$ and fibrant objects $i^{-1}(\F_{\B})$ (the notation denotes preimages of the trivial and fibrant objects in $\B$ under the right adjoint $i$). In case this model on $\A$ exists we denote it by $\A_{\pi}$. Dually, we may speak of the \textit{left-lifted} abelian model structure on $\A$, along the left adjoint $i$, which we denote by $\A_{\iota}$. Note that when these model structures on $\A$ exist, they share the same class of trivial objects, which is $i^{-1}(\W_{\B})$, thus they share the same homotopy category, which we denote by $\mathsf{Ho}(\A)_{\pi/\iota}$. 
Transferring model structures along a left or right adjoint is a deeply studied topic in the theory of model categories and in Section \ref{sec:transfer} we translate certain classic results to the realm of abelian model structures and complete cotorsion pairs. In general, the existence of these transferred model structures is a non-trivial thing to check and heavily depends on context. However, the situation is better when the model category to be transferred is projective or injective (just like in the setup we use above). 
 
In our main result, Theorem \ref{thm:main_lift}, we investigate when, under the above setup, the existence of the left and right model structures on $\mathcal{A}$ imply the existence of a recollement of the associated Quillen homotopy categories
\begin{equation}
\label{eq:dia_in_main_intro}
\xymatrix@C=4pc{
\Ho(\A)_{\pi/\iota}  \ar@<0.0ex>[r] |-{\, \textbf{R}i\cong\textbf{L}i\,}   &  
\Ho(\B)_{\mathsf{\tiny{proj/inj}}}   \ar@/^1.4pc/[l]^-{\textbf{R}p}    \ar@/_1.4pc/[l] _-{\textbf{L}q}  \ar[r] |-{\, \textbf{R}e\cong\textbf{L}e \,}&
\Ho(\C)_{\mathsf{proj/inj}}. \ar@/_1.4pc/[l]  _-{\textbf{L}l}  \ar@/^1.4pc/[l]^-{\textbf{R}r}
  }
  \nonumber
\end{equation}
The point here is to study canonically induced functors $\mathsf{Ho}(\A)_{\pi}\rightarrow \ker (\mathbf{Le})$ and $\mathsf{Ho}(\A)_{\iota}\rightarrow \ker (\mathbf{Re})$  between the homotopy categories of the transferred models on $\A$ and the kernels of the left and right derived functors of $e$. When these maps are fully faithful we speak of \textit{derived embeddings} (Definition~\ref{defn:derembedding}), which in the classical case of chain complexes coincide with \textit{homological embeddings}, as studied by the second author in \cite{Psa_homological_theory}, see Proposition~\ref{prop:hom_embed}. In Section~\ref{sec:applications} we provide applications of the machinery developed for Gorenstein triangular matrix algebras and for $n$--morphism categories over Iwanaga-Gorenstein rings. An example is also given of a ring epimorphism which induces an adjoint triple between the stable category of Gorenstein projective modules over an Iwanaga-Gorenstein ring and the associated transferred homotopy category (Example~\ref{ex:Gor_morphism_triple}).

We end the introduction with a description of the contents of the paper. In Section~\ref{model_cotorsion} we recall several preliminaries from abelian model structures and cotorsion pairs that are used extensively throughout the paper. We also briefly recall the theory of Quillen functors and in Proposition \ref{prop:left_Quillen_triang} we prove that the total left derived functor of a left Quillen functor between hereditary abelian model structures is a triangulated functor. In Section~\ref{sec:Quillen_adj_triples} we introduce Quillen adjoint triples (Definition~\ref{def:Quillen_adjoint_triple}) and examine cases of Quillen adjunctions where the induced derived functors are fully faithful (Propositions~\ref{prop:fully_faithful_right_derived} and~\ref{prop:left_right_fully_faithful}). 
In Section \ref{sec:a_double_psadofunctor} we relate Quillen adjoint triples to work of Shulman on a double pseudofunctor from the double category of model categories to the double category of categories (Theorem~\ref{thm:double_vs_triple}). 

In Section~\ref{sec:transfer} we study transfer of abelian model structures along a left or a right adjoint. More precisely, in Theorem~\ref{thm:right_lifting} we provide sufficient conditions for the existence of the right-lifting of abelian model structures. The conditions are in terms of certain preservation properties of the involved functors and of an acyclicity condition.  Also, in Theorem~\ref{thm:left_lifting} we treat the existence of left-lifting. The section ends with an example of transferring the projective (or injective) model structure in the category of complexes (Propositions \ref{prop:transfer_trivial_proj}, \ref{prop:transfer_trivial_inj}).   
Section~\ref{sec:Lifting_recoll} is devoted to our lifting result on recollements (Theorem~\ref{thm:main_lift}). Finally, Section~\ref{sec:applications} contains three applications of our main results. We construct recollements of stable categories of Cohen-Macaulay modules over triangular matrix algebras (Theorem~\ref{thm:triangular_lifting}) and over the $n$-morphism category of an abelian category (Theorem~\ref{thm:n-lifting}), and we recover the result of Cline, Parhall, and Scott on recollements of derived module categories (Corollary~\ref{cor:CPS}).

\section{Abelian model structures and complete cotorsion pairs}
\label{model_cotorsion}
In this section we recall fundamental concepts concerning abelian model structures. The references for this material include Hovey \cite{hoveybook,hovey2002}, Gillespie \cite{Gil2011} and the survey article of \v{S}{\fontencoding{T1}\selectfont \v{t}}ov\'i\v{c}ek  \cite{Stoviceksurvey}. Proposition \ref{prop:left_Quillen_triang} seems to be new.

\begin{definition}
\label{wfs}
A \textbf{weak factorization system} in a category $\M$ is a pair of classes of morphisms $(\clL,\R)$ such that the following hold:
\begin{itemize}
\item[(i)] Every morphism $f$ in $\M$, admits a factorization as $f=r\circ l$, where $l\in\clL$ and $r\in\R$.
\item[(ii)] $\clL^{\square}=\R$ and $^{\square}\R=\clL$.
\end{itemize}
Here $\clL^{\square}$ denotes the class of morphisms in $\M$ that have the \textit{right lifting property} with respect to all morphisms in the class $\clL$, that is, $g\in\clL^{\square}$ if and only if given the solid part of a diagram as below
\begin{displaymath}
\label{eq:dia}
 \xymatrix@C=2pc{
\cdot \ar[d]_-{l} \ar[r] &  \cdot \ar[d]^{g} \\
\cdot \ar[r] \ar@{-->}[ur]^-{\delta} & \cdot
 }
\end{displaymath}
with $l\in\clL$, there exists a dotted diagonal $\delta$ such that the two triangles commute. Similarly, the class $^{\square}\R$ contains all the morphisms in $\M$ that have the \textit{left lifting property} with respect to every morphism in $\R$.

We call a weak factorization system \textbf{functorial} if its associated factorizations, as in (i), are functorial in $f$.
\end{definition}

The following two definitions are suitable adjustments of the classical notions in the context of abelian categories.

\begin{definition}
An \textbf{abelian weak factorization system} in an abelian category $\M$ is a weak factorization system $(\clL,\R)$ in $\M$ such that:
\begin{itemize}
\item[(i)] A morphism $f$ belongs in $\clL$ if and only if $f$ is a monomorphism and $0\rightarrow\coker (f)$ is in $\clL$.
\item[(ii)] A morphism $f$ belongs in $\R$ if and only if $f$ is an epimorphism and $\ker(f)\rightarrow 0$ is in $\R$.
\end{itemize}
\end{definition}

\begin{definition}
\label{dfn:abelian_model_structure}
Let $\M$ be a bicomplete abelian category. An \textbf{abelian model structure} on $\M$ consists of three classes of morphisms $\mathrm{cof}_{\M}, \mathrm{weak}_{\M}$, and $\mathrm{fib}_{\M}$ in $\M$, such that:
\begin{itemize}
\item[(i)] The pairs $(\mathrm{cof}_{\M},\mathrm{weak}_{\M}\cap\mathrm{fib}_{\M})$ and $(\mathrm{cof}_{\M}\cap\mathrm{weak}_{\M},\mathrm{fib}_{\M})$ are functorial abelian weak factorizations systems in $\M$.
\item[(ii)] The class $\mathrm{weak}_{\M}$ is closed under retracts and satisfies the 2-out-of-3 property with respect to compositions, i.e., if two of $f$, $g$ or $f\circ g$ (when defined) belong in $\mathrm{weak}_{\M}$, then so does the third.
\end{itemize}

If we remove the word ``abelian'' from this definition we obtain the usual concept of a model category, cf. \cite[Def.~1.1.4]{hoveybook}. However, we point out that there is a version of the definition of a model category that assumes only the existence of \textit{finite} limits and colimits in $\M$, see \cite[Def.~7.7]{Joyal}, or even less  assumptions, as in \cite[Def.~6.1]{Stoviceksurvey}, where functoriality is not part of the definition. For the main results in this paper, all model categories will be bicomplete abelian and cofibrantly generated; thus Definition \ref{dfn:abelian_model_structure} serves us well. 
 
We call the morphisms in $\mathrm{cof}_{\M}, \mathrm{fib}_{\M}$, and $\mathrm{weak}_{\M}$, \textit{cofibrations}, \textit{fibrations}, and \textit{weak equivalences} respectively. Morphisms in $\mathrm{cof}_{\M}\cap\mathrm{weak}_{\M}$ and $\mathrm{weak}_{\M}\cap\mathrm{fib}_{\M}$ are called \textit{trivial cofibratrions} and \textit{trivial fibrations} respectively. An object $M$ in $\M$ such that $0\rightarrow M$ or $M\rightarrow 0$ is a weak equivalence, is called a \textit{trivial} object. Also, $M$ is called \textit{(trivially) cofibrant} if $0\rightarrow M$ is a (trivial) cofibration and \textit{(trivially) fibrant} if $M\rightarrow 0$ is a (trivial) fibration.
\end{definition}

Unravelling the definitions, we observe that in an abelian model structure the (trivial) cofibrations in $\M$ are monomorphisms with (trivially) cofibrant cokernel, and that the (trivial) fibrations in $\M$ are epimorphisms with (trivially) fibrant kernel.

Next, we recall how the above notions are related to cotorsion pairs.

\begin{definition}
\label{dfn:cot_pairs}
A \textbf{cotorsion pair} $(\class A,\class B)$ in an abelian category $\class M$, consists of two full subcategories $\class A$ and $\class B$ such that the following hold:
\begin{eqnarray}
\class B=\rightperp{\class A}:=\{M\in\class M\,\,|\,\, \forall A\in\class A,\, \Ext^{1}_{\class M}(A,M)=0\},  \nonumber \\
\class A=\leftperp{\class B}:=\{M\in\class M\,\,|\,\, \forall B\in\class B,\, \Ext^{1}_{\class M}(M,B)=0\}. \nonumber
\end{eqnarray}
Here $\Ext^{1}_{\class M}(-,-)$ denotes the Yoneda Ext-bifunctor. A cotorsion pair $(\class A,\class B)$ in $\class M$ is called \textbf{complete} if for each object $M$ in $\class M$, there exists a short exact sequence $0\rightarrow B\rightarrow A\rightarrow M\rightarrow 0$ with $A\in\class A$ and $B\in\class B$, and also a short exact sequence $0\rightarrow M\rightarrow B'\rightarrow A'\rightarrow 0$ with $A'\in\class A$ and $B'\in\class B$. It is called \textbf{hereditary} if for all $A\in\class A$, $B\in\class B$ and $i\geqslant 1$, we have $\Ext_{\class M}^{i}(A,B)=0$.
\end{definition}

The following fact is essentially due to Hovey \cite{hovey2002}. Here we use the language in which it is stated in \cite{Stoviceksurvey} (where it is given in a more general context).

\begin{fact}\textnormal{(\cite[Thm.~5.13]{Stoviceksurvey})}
\label{fact:wfs_vs_cot_pairs}
Let $\M$ be an abelian category. The mappings
\[(\clL,\R)\longmapsto (\coker(\clL),\ker(\R)) \,\,\,\,\,\,\, \mbox{and} \,\,\,\,\,\,\, (\A,\B)\longmapsto (\mathrm{Mono}(\A),\mathrm{Epi}(\B))\]
define mutually inverse bijections between (functorial) abelian weak factorization systems in $\M$ and (functorially) complete cotorsion pairs in $\M$.
Here $\coker(\clL)$ denotes the class of objects in $\M$ which are isomorphic to $\coker(f)$ for some morphism $f$ in $\clL$, and dually for $\ker(\R)$. Also, $\mathrm{Mono}(\A)$ is the class of monomorphisms in $\M$ with cokernel in $\A$, and dually, $\mathrm{Epi}(\A)$ is the class of epimorphisms in $\M$ with kernel in $\B$. 
\end{fact}

We also spell out Hovey's original result from \cite{hovey2002}. 

\begin{fact}
\textnormal{(\cite[Thm.~2.2]{hovey2002})}
Consider an abelian model structure on a bicomplete abelian category $\M$ where $\C_{\M}, \F_{\M}$ and $\W_{\M}$ denote the classes of cofibrant, fibrant and trivial objects respectively. Then there exist (functorially) complete cotorsion pairs $(\class C_{\M}\cap\class W_{\M},\class F_{\M})$ and $(\class C_{\M},\class W_{\M}\cap\class F_{\M})$ in $\class M$, where $\class W_{\M}$ is closed under summands and is a thick subcategory of $\class M$.\footnote{Thick here means that for any short exact sequence in $\class M$, if two of its terms belong in $\class W$, then so does the third.}

Conversely, if there exist classes $\class C_{\M}, \class W_{\M}$ and $\class F_{\M}$ of objects in $\class M$, where $\class W_{\M}$ is thick and closed under summands, and (functorially) complete cotorsion pairs $(\class C_{\M}\cap\class W_{\M},\class \F_{\M})$ and $(\class C_{\M},\class W_{\M}\cap\class F_{\M})$ in $\class M$, then $\class M$ is an abelian model structure where $\class C_{\M}, \class W_{\M}$ and $\class F_{\M}$ are the classes of cofibrant, fibrant and trivial objects respectively.
\end{fact}

\begin{definition}
\label{def:Hovey_triples}
Given an abelian model structure on a bicomplete abelian category $\M$ where $\C_{\M},\F_{\M}$ and $\W_{\M}$ denote the classes of cofibrant, fibrant and trivial objects respectively, we abbreviate by saying that $(\C_{\M},\W_{\M},\F_{\M})$ is a \textbf{Hovey triple} on $\M$.  In case the associated cotorsion pairs that define $\M$ are hereditary, we say that $(\C_{\M},\W_{\M},\F_{\M})$ is a \textbf{hereditary Hovey triple} on $\M$.
\end{definition}

\subsection*{Homotopy categories}
\label{subsec:homotopy_cats}
Let $\M$ be an abelian model structure with Hovey triple $(\C_{\M},\W_{\M},\F_{\M})$. By definition, the homotopy category $\Ho(\M)$, is the localization of $\M$ with respect to the class of weak equivalences. We note that from \cite[Prop.~2.3]{Gilsurvey} the class of weak equivalences $\mathrm{weak}_{\M}$ consists precisely of those morphisms in $\M$ that factor as a monomorphism with trivial cokernel followed by an epimorphism with trivial kernel. Henceforth we denote weak equivalences by $\sim$.

From the general theory of model categories, we know that $\Ho(\M)$ is equivalent to a quotient category $\C_{\M}\cap\F_{\M}/r$\, , where\, $r$ is a certain ``homotopy'' equivalence relation, see for instance \cite[Theorem.~1.2.10]{hoveybook}.

Now, assume that $\M$ is a hereditary abelian model structure. In that case, from Gillepsie \cite{Gil2011} we know that the category $\M_{cf}:=\C_{\M}\cap\F_{\M}$ is additive Frobenius with projective-injective objects equal to the class $\C_{\M}\cap\F_{\M} \cap\W_{\M}$, and that the equivalence relation $r$ identifies two parallel maps if they factor through a projective-injective object (see also \cite[Prop.~4.2]{Gilsurvey}). Thus we obtain that $\M_{cf}/r$ is the same as $\underline{\M}_{cf}$, where the latter denotes the stable category of the Frobenius category $\M_{cf}$. In total, we have that the composite
\begin{displaymath}
 \xymatrix@C=2pc{
 \underline{\M}_{cf}=\M_{cf}/r \ar[r]^-{\mathrm{can}}   &  \Ho(\M_{cf}) \ar[r]^-{\Ho(\mathrm{inc})} & \Ho(\M)
 }
\end{displaymath}
is an equivalence of categories. 

Stable categories of Frobenius categories are triangulated \cite[I.2]{Happelbook}, hence the above equivalence imposes a triangulated structure on $\Ho(\M)$, with its distinguished triangles being isomorphic to the images of the distinguished triangles in $\underline{\M}_{cf}$ via the aforementioned equivalence. We now describe a standard method to produce a triangle in $\Ho(\M)$ starting from the hereditary abelian model structure $\M$. To that end, following \cite[Definition~6.16]{Stoviceksurvey}, we call a diagram of arrows and morphisms $X\xrightarrow{u} Y\xrightarrow{v} Z\xrightarrow{w} \Sigma X$ in $\M$ a \textbf{cofiber sequence} if it fits to a commutative diagram
\begin{displaymath}
 \xymatrix@C=2pc{
0 \ar[r] & X\ar[d]^-{u}\ar[r]   & W\ar[r] \ar[d]  & \Sigma X\ar@{=}[d]\ar[r] & 0 \\
0 \ar[r]  & Y \ar[r]^-{v} & Z\ar[r]^-{w} & \Sigma X \ar[r] & 0,
 }
\end{displaymath}
where the rows are short exact sequences in $\M$ and $W\in\W_{\M}$. 

Any given cofiber sequence in $\M$ as above defines a triangle in $\Ho(\M)$. Indeed, using the completeness of the cotorsion pairs $(\C_{\M},\W_{\M}\cap\F_{\M})$ and $(\C_{\M}\cap\W_{\M},\F_{\M})$, we may construct (see \cite[Thm.~6.21]{Stoviceksurvey}) a cofiber sequence \[X'\xrightarrow{u'} Y'\xrightarrow{v'} Z'\xrightarrow{w'} \Sigma X'\] of fibrant and cofibrant objects in $\M$, which is a triangle in the triangulated category $\underline{\M}_{cf}$ and is isomorphic in $\Ho(\M)$ to the initial cofiber sequence. In fact, every triangle in $\Ho(\M)$ is induced from a cofiber sequence of fibrant and cofibrant objects; this essentially uses the triangulated structure on $\underline{\M}_{cf}$.

\begin{remark}
In any model category (with functorial weak factorization systems) there exist cofibrant and fibrant replacement functors, see for instance \cite[page~5]{hoveybook}.  In case $\M$ is an abelian model structure with Hovey triple $(\C_{\M},\W_{\M},\F_{\M})$, functoriality of the complete cotorsion pair $(\C_{\M},\W_{\M}\cap\F_{\M})$ defines a functor $Q_{\M}\colon\M\rightarrow\C_{\M}$ which is called \textbf{cofibrant replacement}. 
In fact, for any object $X$ in $\M$, there exists a short exact sequence, natural in $X$,  
\begin{equation}
 \xymatrix@C=2pc{
0\ar[r] & K \ar[r]  &  Q_{\M}X \ar[r]^-{q^{\M}_{X}}_{\sim} & X  \ar[r] & 0,
}
\nonumber
\end{equation}
with $Q_{\M}X$ in $\mathcal{C}_{\mathcal{M}}$ and $K$ in $\F_{\M}\cap\W_{\M}$. Note that the morphism $q^{\mathcal{M}}_{X}$ is a trivial fibration. 
Similarly, a \textbf{fibrant replacement} functor $R_{\M}\colon\M\rightarrow\F_{\M}$ is provided by functoriality of the complete cotorsion pair $(\C_{\M}\cap\W_{\M},\F_{\M})$, and for any $X$ in $\M$ there exists a short exact sequence which is natural in $X$, 
\begin{equation}
 \xymatrix@C=2pc{
0\ar[r] & X \ar[r]^-{r^{\M}_{X}}_-{\sim}  &  R_{\M}X \ar[r] & C  \ar[r] & 0,
}
\nonumber
\end{equation}
where $R_{\M}X\in\F_{\M}$, $C\in\C_{\M}\cap\W_{\M}$ and the morphism $r^{\M}_{X}$ is a trivial cofibration. When things are clear from context, we might drop the subscript $``\M''$ from the functors $Q_{\M}$ and $R_{\M}$.
\end{remark}

\subsection{Quillen functors} 

We briefly recall some generalities on Quillen functors. 

A functor between model categories is called \textbf{left Quillen} if it is a left adjoint and maps cofibrations to cofibrations and trivial cofibrations to trivial cofibrations. Dually, it is called \textbf{right Quillen} if it is a right adjoint and maps fibrations to fibrations and trivial fibrations to trivial fibrations. An (ordinary) adjunction between model categories where the left adjoint is left Quillen, or equivalently, from \cite[Lem.~1.3.4]{hoveybook}, the right adjoint is right Quillen, is called a \textbf{Quillen adjunction}.

We will assume that the definition of total left (or right) derived functors is known, see \cite[\S8.2]{Stoviceksurvey} for a discussion that also fits into the context of this paper. We just mention that if 
 $F\colon \M\rightarrow \C$ is a left Quillen functor between model categories, then its total left derived functor $\textbf{L}F\colon\Ho(\M)\rightarrow\Ho(\C)$ can be computed on objects as follows: For every $X$ in $\Ho(\M)$, take the cofibrant replacement $QX$ of $X$ and define $\mathbf{L}F(X):=F(QX)$. 
Similarly, if $G\colon \M\rightarrow \C$ is a right Quillen functor between model categories, its total right derived functor $\textbf{R}G\colon\Ho(\M)\rightarrow\Ho(\C)$ can be computed on objects as follows: For every $X$ in $\Ho(\M)$, take the fibrant replacement $RX$ of $X$ and define $\mathbf{R}G(X):=F(RX)$. 

We close this section with the following fundamental observation.

\begin{proposition}
\label{prop:left_Quillen_triang}
Let $F\colon\A\rightarrow\B$ be a left Quillen functor between hereditary abelian model structures. Then the total left derived functor $\textbf{L}F\colon\Ho(\A)\rightarrow \Ho(\B)$ is a triangulated functor.
\end{proposition}

\begin{proof}
We want to prove that the functor $\textbf{L}F$ maps distinguished triangles in $\Ho(\A)$ to distinguished triangles in $\Ho(\B)$. Let $\Delta$ be a distinguished triangle in $\Ho(\A)$. So $\Delta$ is isomorphic in $\Ho(\A)$ to the image of a standard triangle $X\xrightarrow{u} Y\xrightarrow{v} Z\xrightarrow{w} \Sigma X$ in $\underline{\A}_{cf}$ which sits in a pushout diagram in $\A$,
\begin{displaymath}
 \xymatrix@C=2pc{
0 \ar[r] & X\ar[d]^-{u}\ar[r]^-{a}   & W\ar[r]^-{b} \ar[d]  & \Sigma X\ar@{=}[d]\ar[r] & 0 \\
0 \ar[r]  & Y \ar[r]^-{v} & Z\ar[r]^-{w} & \Sigma X \ar[r] & 0
 }
\end{displaymath}
having its rows cofibrations and $W$ a projective-injective object in the Frobenius category $\C_{\M}\cap\F_{\M}$, so $W\in\C_{\A}\cap\W_{\A}\cap\F_{\A}$. Since all the objects appearing in $\Delta$ are in particular cofibrant, the action of $\textbf{L}(F)$ on $\Delta$ produces a diagram in $\Ho(\B)$,
\begin{displaymath}
 \xymatrix@C=2pc{
\textbf{L}F(\Delta)\colon &  F(X)\ar[r]^-{F(u)} & F(Y) \ar[r]^-{F(v)} & F(W) \ar[r]^-{F(w)} & F(\Sigma X)
 }
\end{displaymath}
which we want to be a distinguished (i.e., isomorphic to a standard) triangle in $\Ho(\B)$. To this end we consider the following diagram in $\B$,
\begin{equation}
\label{eq:front_wall}
 \xymatrix@C=2pc{
0 \ar[r] & FX\ar[d]^-{Fu}\ar[r]^-{Fa}   & FW\ar[r]^-{Fb} \ar[d]  & \Sigma FX\ar@{=}[d]\ar[r] & 0 \\
0 \ar[r]  & FY \ar[r]^-{Fv} & FZ\ar[r]^-{Fw} & \Sigma  FX\ar[r] & 0.
 }
\end{equation}
Note that this is still a pushout (since $F$ is a left adjoint), its rows are cofibrations and $FW$ belongs in $C_{\B}\cap\W_{\B}$ (since $F$ is left Quillen).

Since $\B$ is a hereditary abelian model structure, a version of the Horseshoe Lemma holds \cite[Lem.~6.20]{Stoviceksurvey}, so we may paste together a fibrant replacement of $F(X)$ and a fibrant replacement of $\Sigma FX$ to obtain a fibrant replacement of $FW$, pictorially:
\begin{equation}
\label{eq:top_wall}
 \xymatrix@C=2pc{
 & J \ar[r]^-{a'}   & J' \ar[r]^-{b'}   & J'' \\
F\! X \, \ar@{>->}[ur]^-{j}_{\sim}\ar[r]_-{Fa}   & F\! W \, \ar[r]_-{Fb} \ar@{>->}[ur]^-{j'}_{\sim}  & \Sigma\! FX \, \ar@{>->}[ur]^-{j''}_{\sim} &
 }
\end{equation}
where $j,j'$ and $j''$ are monomorphisms with cokernel in $\W_{\B}\cap\C_{\B}$. We observe that  by construction the middle morphism $j'$ has its target $J'$ in the class of projective-injective objects $\C_{\B}\cap\F_{\B}\cap\W_{\B}$. Moreover, we consider a fibrant replacement $h\colon F(Y)\rightarrow H$ of $F(Y)$, i.e., $h$ is a monomorphism with cokernel in $\W_{\B}\cap\C_{\B}$. We construct the following commutative diagram:

    \begin{equation}
    \label{eq:big_diagram}
    \begin{gathered}
      \xymatrix@!=0.5pc{ {} & J \, \ar'[d][dd]_-{\lambda}
          \ar@{>->}[rr]^-{a'} & & J' \ar'[d][dd]_-{} 
          \ar@{->>}[rr]^-{b'} 
          & & J'' \ar@{=}[dd]
          \\
          F\!X \,
          \ar@{>->}[ur]^(0.45){j}
          \ar[dd]_-{F\!u}
          \ar@{>->}[rr]^(0.70){F\!a} & & 
          F\!W
          \ar@{>->}[ur]^(0.45){j'}
          \ar[dd]_(0.70){F\!c}
          \ar@{->>}[rr]^(0.70){F\!b} & & \Sigma\! FX
          \ar@{>->}[ur]^(0.45){j''}
          \ar@{=}[dd]
          \\
          {} & H \, \ar@{>->}'[r]^-{a''}[rr] & & 
          H' \ar@{->>}'[r]^-{b''}[rr] 
          & & J''
          \\
          F\!Y \,
          \ar@{>->}[ur]_(0.6){h}
          \ar@{>->}[rr]_-{F\!v} & & F\!Z
          \ar@{>..>}[ur]_(0.6){h'}
          \ar@{->>}[rr]_-{F\!w} & & \Sigma\! FX
          \ar@{>->}[ur]_(0.6){j''}
        }
    \end{gathered}
    \end{equation}
    
Here the top wall is just diagram (\ref{eq:top_wall}), while the front wall is diagram (\ref{eq:front_wall}). The morphism $\lambda$ is obtained since $j$ is a monomorphism with cokernel in $\W_{\B}\cap\C_{\B}$ and $H\in\F_{\B}=\rightperp{(\W_{\B}\cap\C_{\B})}$. The back wall is defined as the pushout of $H\longleftarrow J \longrightarrow J'$. It now follows from a diagram chase that there exists a dotted map $h'$ which makes the whole diagram commutative. 
Also, the snake lemma applied to the bottom wall shows that $h'$ is a monomorphism with cokernel in the class $\W_{\B}$. 

Part of diagram (\ref{eq:big_diagram}) is the commutative diagram:
\begin{equation}
\label{eq:iso_triangles}
 \xymatrix@C=2pc{
FX\ar[d]_-{j}\ar[r]^-{Fu}   & FY\ar[r]^-{Fv} \ar[d]^-{h}  & FZ\ar[d]^-{h'} \ar[r]^-{Fw} & \Sigma FX\ar[d]^-{\Sigma j} \\
J \ar[r]^-{\lambda} & H\ar[r]^-{a''} & H'\ar[r]^-{b''} & \Sigma J
 }
\end{equation}
where by construction the maps $j, h$ and $h'$ are isomorphisms in $\Ho(\B)$. Thus we have proved that the diagram $\textbf{L}F(\Delta)$ is  isomorphic in $\Ho(\B)$ to the bottom row of (\ref{eq:iso_triangles}), which is a standard triangle since it sits at the back wall of diagram (\ref{eq:big_diagram}).
\end{proof}

\begin{corollary}
\label{cor:triang_Quillen_adj}
A Quillen adjunction $F\colon \A\rightleftarrows\B\colon G$ between hereditary abelian model structures induces an adjunction of triangulated functors $\textbf{L}F\colon\Ho(A)\leftrightarrows\Ho(\B)\colon\textbf{R}G$.
\end{corollary}

\begin{proof}
It is standard that the induced adjunction $\textbf{L}F\colon\Ho(A)\leftrightarrows\Ho(\B)\colon\textbf{R}G$ exists. From Proposition \ref{prop:left_Quillen_triang} we obtain that $\textbf{L}F$ is a triangulated functor. Now \cite[Lemma~5.3.6]{Nee} implies that its right adjoint $\textbf{R}G$ is also triangulated.
\end{proof}

We also record the following conceptual observation. For unexplained terminology in what follows we refer to \cite[Chapter~1, 1.4]{hoveybook}. We recall that there exists a 2-category $\mathit{Model_{L}}$ of model categories, left Quillen functors, and natural transformations, and dually, there exists a 2-category $\mathit{Model_{R}}$ of model categories, right Quillen functors, and natural transformations. We also denote by $\mathit{Cat}$ the 2-category of categories. From the discussion in \textit{loc.cit} (see also \cite[Thm.~3.2]{Shulman}) it follows that there exist pseudofunctors
\[\textbf{L}\colon\mathit{Model_{L}}\rightarrow \mathit{Cat} \,\,\,\,\,\mbox{and}\,\,\,\,\,\, \textbf{R}\colon \mathit{Model_{R}}\rightarrow \mathit{Cat}\]
which map a model category to its homotopy category and a left (resp., right) Quillen functor to its left (resp., right) derived functor. 

We have the following consequence of Proposition \ref{prop:left_Quillen_triang}.

\begin{corollary}
The pseudofunctor $\textbf{L}$ restricts to a pseudofunctor $\textbf{L}\colon \mathit{hAbModel_{L}}\rightarrow \mathit{Triang}$ between the 2-category $\mathit{hAbModel}$ of hereditary abelian model structures, left Quillen functors, and natural transformations and the 2-category $\mathit{Triang}$ of triangulated categories, triangulated functors, and natural transformations. A similar statement holds for the pseudofunctor $\textbf{R}$.
\end{corollary}

\section{Quillen adjoint triples}
\label{sec:Quillen_adj_triples}

We introduce the concept of a Quillen adjoint triple, which is implicit in work of Shulman \cite{Shulman}\footnote{In addition, in the work \cite{Jordan} Quillen adjoint triples are used in a more strict setup.}.
The discussion applies to general model categories and we do not need to restrict to the abelian case. In Section \ref{sec:a_double_psadofunctor} we explain how Quillen adjoint triples are related to the more abstract world of \cite{Shulman}. Also, in subsection \ref{subsec:fully_faith} below we study when certain derived functors are fully faithful.

In what follows, if a functor $f$ is left adjoint to a functor $f'$, we write $f\dashv f'$.

\begin{definition}
\label{def:Quillen_adjoint_triple}
Assume that we are given two (bicomplete) categories $\A$ and $\B$ such that
$\A$ admits two model structures $\A_1, \A_2$ and a Quillen equivalence $f\colon \A_1 \rightleftarrows \A_2\colon f'$ (with $f\dashv f'$), and similarly, 
$\B$ admits two model structures $\B_1, \B_2$ and a Quillen equivalence $g\colon \B_1 \rightleftarrows \B_2\colon g'$ (with $g\dashv g'$).
A \textbf{Quillen adjoint triple} is defined as the data of two Quillen adjunctions
\begin{equation}
  \xymatrix@C=1pc{
  \A_1 \ar@/_1.2pc/[rr]_-{i} & \bot & \B_{1} \ar@/_1.2pc/[ll]_-{q}
  } 
  \\ \,\,\,\,\,\,\,\,\,\,\,\mbox{and}\,\,\,\,\,\,\,\,\,\,\,
  \xymatrix@C=1pc{
  \A_2 \ar@/^1.2pc/[rr]^-{i'} & \bot & \B_{2} \ar@/^1.2pc/[ll]^-{p}
  } 
  \nonumber
  \end{equation}
and a natural transformation $\alpha\colon g\circ i\Rightarrow i'\circ f,$
\begin{equation}
\label{eq:alpha_in_def_of_Quillen_adj}
  \xymatrix@C=1pc{
\A_1 \ar[d]_-{f}  \ar[r] ^-{i} \drtwocell\omit{\alpha}   &  \B_{1} \ar[d] ^-{g} \\
\A_{2} \ar[r]_-{i'} &  \B_{2},
  }\!
\end{equation}
such that there exists a natural isomorphism $\mathbf{ho}(\alpha)\colon\mathbf{L}g\circ\mathbf{R}i\Rightarrow\mathbf{L}i'\circ\mathbf{L}f,$
\begin{equation}
  \xymatrix@C=3pc{
\Ho(\A_1) \ar[d]_-{\textbf{L}f}  \ar[r] ^-{\textbf{R}i} \drtwocell\omit{\,\,\,\,\,\,\,\,\,\,\,\,\mathbf{ho}(\alpha)}   &  \Ho(\B_{1}) \ar[d] ^-{\mathbf{L}g} \\
  \Ho(\A_{2}) \ar[r]_-{\textbf{L}i'} &  \Ho(\B_{2}),  
  }\!
  \nonumber
\end{equation}
represented by the composite of the following zigzag in $\Ho(\B_2)$, where $R$ and $Q$ denote fibrant and cofibrant replacement functors respectively and it is clear from context in which model we apply them,

\begin{equation}
\label{eq:composite_map_triples}
\xymatrix@C=3pc{
gQ(iR) & 
gQiQR  \ar[l]_-{\sim} \ar[r]^-{gq^{\B_1}_{i QR}} &
(g\circ i)QR \ar[d]^-{\alpha_{QR}} & &   \\
& & 
(i'\circ f) QR &
 (i'\circ f)Q. \ar[l]_-{\sim}
}\!
\end{equation}
We denote a Quillen adjoint triple by $\mathrm{Q_{tr}}(\A_{1/2},\B_{1/2},\alpha)$.

We explain in more detail the morphisms in diagram (\ref{eq:composite_map_triples}). Fix an object $X$ in $\A_1$. The vertical map $\alpha_{QRX}$ is just the natural transformation $\alpha$ applied on a fibrant-cofibrant replacement of $X$. The upper right map occurs after applying the right Quillen functor $g$ on the cofibrant replacement map $q^{\B_1}_{i QRX}$. 
The upper left map occurs as follows:  We first consider the following commutative square in $\B_{1}$, which is natural in $X$, 
\begin{equation}
\label{eq:explanation_in_def}
  \xymatrix@C=1pc{
iRX     & &  iQRX \ar[ll]_-{iq^{\A_1}_{RX}}^-{\sim} \\
QiRX \ar[u]^-{\sim} &  & QiQRX \ar[u]^-{\sim}   \ar[ll]^-{Q\, iq^{\A_1}_{RX}}_-{\sim} 
  }\!
\end{equation}
where the map on top is a trivial fibration, since the right Quillen $i$ maps the trivial fibration $q^{\A_1}_{RX}$ to a trivial fibration, and the bottom map is a cofibrant replacement of the top map. We then apply the left Quillen functor $g$ on the bottom map of (\ref{eq:explanation_in_def}), which by \cite[Lemma~1.1.12]{hoveybook} gives a weak equivalence; this is the upper left map in (\ref{eq:composite_map_triples}). For the bottom right map in (\ref{eq:composite_map_triples}), we first consider the trivial cofibration $r_{X}^{\A_1}\colon X\rightarrow RX$ and then take its cofibrant replacement, $Qr_{X}^{\A_1}$, which again by \textit{loc.cit} is mapped to a weak equivalence by the left Quillen functor $i'\circ f$.
\end{definition}

\begin{remark}
\label{rem:fibrant-cofibrant objects}
For any fibrant-cofibrant object $X$ in $\A_1$, the composite of zig-zags (\ref{eq:composite_map_triples}) in Definition \ref{def:Quillen_adjoint_triple} reduces to the composite
\begin{equation}
\label{eq:nat_map_cf}
  \xymatrix@C=2pc{
gQ_{\B_{1}}(iX) \ar[r]^-{gq^{\B_{1}}_{iX}} & (g\circ i)X \ar[r]^-{\alpha_{X}} & (i'\circ f)X.
}\!
\end{equation}
Indeed, we observe that for $X$ fibrant-cofibrant 
the top map in diagram (\ref{eq:explanation_in_def}) is the identity map and the bottom map is $Q(\mathrm{id}_{iX})=\mathrm{id}_{QiX}$; hence the upper left map in (\ref{eq:composite_map_triples}) is the identity map $\mathrm{id}_{gQiX}$. Similarly, one can see that the bottom right map in (\ref{eq:composite_map_triples}) is $\mathrm{id}_{(i'\circ f)X}$. Now, it is not hard to prove that the composite of zig-zags from  (\ref{eq:composite_map_triples}) is an isomorphism in $\Ho(\B_2)$ if and only if the composite in (\ref{eq:nat_map_cf}) is, but we will deduce this directly from work of Shulman, see~Corollary~\ref{cor:Shulman}.
\end{remark}

\begin{remark}
\label{rem:easy_case}
In the context of Definition \ref{def:Quillen_adjoint_triple}, in case the functors $f,f',g,g'$ are identities, $i=i', \alpha=\mathrm{id}$ and $\B_1$ shares the same weak equivalences with $\B_2$, the zig-zag of maps (\ref{eq:composite_map_triples}) in Definition \ref{def:Quillen_adjoint_triple} is always a natural isomorphism in $\Ho(\B_2)$. Indeed, in this case the right hand side morphism in the composite (\ref{eq:nat_map_cf}) in Remark \ref{rem:fibrant-cofibrant objects} is the identity, while the left hand side morphism is a weak equivalence in $\B_2$ (as the image of the weak equivalence $q_X^{\B_1}$ under $g=\mathrm{id}\colon \B_1\rightarrow\B_2$). This ``trivial'' special case is typical in applications.
\end{remark}

\begin{remark}
\label{rem:mates}
In the context of Definition \ref{def:Quillen_adjoint_triple}, the natural transformation $\alpha\colon g\circ i \Rightarrow i'\circ f$ has a \textit{mate}, $\tilde{\alpha}\colon i\circ f'\Rightarrow g'\circ i'$,
\begin{equation}
  \xymatrix@C=1pc{
\A_1   \ar[r] ^-{i}   &  \B_{1}  \\
\A_{2} \ar[r]_-{i'} \ar[u]^-{f'}  \urtwocell\omit{\tilde{\alpha}}  &  \B_{2},  \ar[u]_-{g'}
  }\!
  \nonumber
\end{equation}
which is the natural transformation defined by the composite
\begin{equation}
 \xymatrix@C=2pc{
 i\circ f' \ar[r]^-{\eta if'} & g'\circ g\circ i\circ f' \ar[r]^-{g'\alpha f'} & g'\circ i'\circ f\circ f' \ar[r]^-{g'i'\epsilon} & g'\circ i',
}
 \end{equation}
where $\eta$ denotes the unit of $g\dashv g'$ and $\epsilon$ the counit of $f\dashv f'$. In complete analogy, the natural transformation $\mathbf{ho}(\alpha)$ also has a mate 
\begin{equation}
  \xymatrix@C=3pc{
\Ho(\A_1)  \ar[r] ^-{\textbf{R}i}    &  \Ho(\B_{1})   \\
  \Ho(\A_{2}) \ar[u]^-{\mathbf{R}f'}  \ar[r]_-{\textbf{L}i'} \urtwocell\omit{\,\,\,\,\,\,\,\,\,\,\,\,\widetilde{\mathbf{ho}(\alpha)}} &  \Ho(\B_{2}).  \ar[u]_-{\textbf{R}g'}
  }\!
  \nonumber
\end{equation}
Since $\mathbf{L}f\dashv \mathbf{R}f'$ and $\mathbf{L}g\dashv \mathbf{R}g'$ are adjoints equivalences, a well known fact implies that $\mathbf{ho}(\alpha)$ is an isomorphism if and only if its mate $\widetilde{\mathbf{ho}(\alpha)}$ is an isomorphism, see for instance \cite[Lemma 2.2]{Shulman}.
\end{remark}

\begin{remark}(\textnormal{Adjoint triples of homotopy categories})
\label{rem:induced_adj_triple}
Given a Quillen adjoint triple $\mathrm{Q_{tr}}(\A_{1/2},\B_{1/2},\alpha)$, there is a natural isomorphism $\mathbf{ho}(\alpha)\colon\mathbf{L}g\circ\mathbf{R}i\Rightarrow\mathbf{L}i'\circ\mathbf{L}f$ where the functors $\mathbf{L}f$ and $\mathbf{L}g$ are equivalences of categories. 
In more detail, we have the following diagram (natural isomorphism $\mathbf{ho}(\alpha)\colon\mathbf{L}g\circ \mathbf{R}i\Rightarrow \mathbf{L}i'\circ\mathbf{L}f$)
\[ 
 \xymatrix@C=3pc{
\Ho(\A_{1}) \ar[d]_{\textbf{L}f}^{\cong}   \ar@<0.0ex>[rr] ^-{ \textbf{R}i }   &&
  \Ho(\B_{1}) \ar[d]^{\textbf{L}g}_{\cong} \ar@/_1.5pc/[ll] _-{\textbf{L}q}  \\
\Ho(\A_{2}) \ar@<0.0ex>[rr] ^-{ \textbf{L}i' }   && \Ho(\B_{2})  \ar@/^1.5pc/[ll] ^-{\textbf{R}p}    
  }
\]
where the functors $\textbf{L}q\dashv \textbf{R}i$ and $\textbf{L}i'\dashv \textbf{R}p$ form adjoint pairs. In particular, we obtain an adjoint triple of homotopy categories
\begin{equation}
\label{eq:adjoint_triple_ab2}
  \xymatrix@C=4pc{
\Ho(\A_{2})   \ar@<0.0ex>[rr] ^-{\textbf{L}i'\, \cong\, \textbf{L}g  \textbf{R}i (\textbf{L}f)^{-1}}   &&
  \Ho(\B_{2}) \ar@/^1.5pc/[ll] ^-{\textbf{R}p}    \ar@/_1.5pc/[ll] _-{\textbf{L}f\textbf{L}q (\textbf{L}g)^{-1}} 
  }\! 
\end{equation}
as well as an adjoint triple of homotopy categories,
\begin{equation}
\label{eq:adjoint_triple_mixed2}
  \xymatrix@C=4pc{
\Ho(\A_{1})   \ar@<0.0ex>[rr] ^-{\textbf{R}i \, \cong\, (\textbf{L}g)^{-1}  \textbf{L}i' \textbf{L}f }   &&
  \Ho(\B_{1}). \ar@/^1.5pc/[ll] ^-{(\textbf{L}f)^{-1}\textbf{R}p \textbf{L}g}    \ar@/_1.5pc/[ll] _-{\textbf{L}q } 
  }\! 
\end{equation}
In case the model categories involved are hereditary abelian, from Corollary \ref{cor:triang_Quillen_adj} we deduce that these are adjoint triples of triangulated categories and triangulated functors. 
\end{remark}

We specialize the situation in the next useful statement. 

\begin{proposition} 
\label{prop:special_triple}
Let $\mathrm{Q_{tr}}(\A_{1/2},\B_{1/2},\alpha)$ be a Quillen adjoint triple, where $f,f',g,g'$ are identity maps, $i=i'$, $\alpha=\mathrm{id}$ and $\A_1$ (resp., $\B_1$) shares the same weak equivalences with $\A_2$ (resp., $\B_2$). Then there exists an adjoint triple of homotopy categories 
\begin{equation}
\label{eq:proinj_triple}
  \xymatrix@C=4pc{
\Ho(\A_{1/2})   \ar@<0.0ex>[rr] ^-{\textbf{R}i\, \cong\,  \textbf{L}i }   &&
  \Ho(\B_{1/2}), \ar@/^1.5pc/[ll] ^-{\textbf{R}p}    \ar@/_1.5pc/[ll] _-{\textbf{L}q} 
  }\! 
\end{equation}
where $\Ho(\A_{1/2})$  denotes the common homotopy category for $\A_1$ and $\A_2$, and similarly, $\Ho(\B_{1/2})$  denotes the common homotopy category for $\B_1$ and $\B_2$.
\end{proposition}

\begin{proof}
Since there are adjoint pairs $\mathbf{L}q\dashv\mathbf{R}i$ and $\mathbf{L}i\dashv \mathbf{R}p$, it suffices to prove that we have a natural isomorphism $\textbf{L}i\, \cong\,  \textbf{R}i$. 
To that end, consider an object $X$ in $\A$ and the following composite, which is natural in $X$,
\begin{equation}
  \xymatrix@C=2pc{
  Q_{\A_2}X \ar[r]^-{q^{\A_2}_{X}} & X \ar[r]^-{r^{\A_1}_{X}} & R_{\A_{1}}X 
  }
  \nonumber
\end{equation}  
where $q^{\A_2}_{X}$ is a cofibrant replacement of $X$ in $\A_2$ and $r^{\A_1}_{X}$ is a fibrant replacement of $X$ in $\A_1$. We apply the functor $i$ to obtain 
\begin{equation}
  \xymatrix@C=2pc{
\mathbf{L}i(X)= i(Q_{\A_2}X) \ar[r]^-{iq^{\A_2}_{X}} & i(X) \ar[r]^-{ir^{\A_1}_{X}} & i(R_{\A_{1}}X)=\mathbf{R}i(X)
  }
  \nonumber
\end{equation}  
Since $i\circ\mathrm{id}\colon \A_1\rightarrow\B_2$ is a left Quillen functor, the map $ir^{\A_1}_{X}$ is a trivial cofibration in $\B_2$, and since $i\circ \mathrm{id}\colon \A_2\rightarrow \B_1$ is a right Quillen functor, the map $iq^{\A_2}_{X}$ is a trivial fibration in $\B_1$. In particular, the maps $ir^{\A_1}_{X}$ and $iq^{\A_2}_{X}$ are weak equivalences in $\B_2$ and $\B_1$ respectively.  Since $\B_2$ and $\B_1$ share the same weak equivalences the result follows.
\end{proof}

\subsection{Fully faithful derived functors}
\label{subsec:fully_faith}
We gather a few results characterizing when certain derived functors are fully faithful. 
\begin{proposition} 
\label{prop:fully_faithful_right_derived}
Let $\A$ and $\B$ be two hereditary abelian model structures with Hovey triples $(\C_{\A},\W_{\A},\F_{\A})$ and $(\C_{\B},\W_{\B},\F_{\B})$ respectively. Let $l\colon\B\leftrightarrows\A\colon r$ be a Quillen adjunction, where $r$ is the right adjoint. The following are equivalent:
\begin{itemize}
\item[(i)] The total right derived functor $\textbf{R}r$ is fully faithful.
\item[(ii)] For all $X,Y\in \A$ and $n\geqslant 1$,\, $\Ext_{\A}^n(QX, RY)\cong \Ext_{\B}^n(QrRX, rRY)$. 
\end{itemize}
In case $r$ is fully faithful, the above statements are also equivalent to:
\begin{itemize}
\item[(iii)] For all $X$ fibrant in $\A$, the weak equivalence $q^{\B}_{rX}\colon Q_{\B}(rX)\rightarrow rX$ is mapped to a weak equivalence under $l$.
\end{itemize}
\end{proposition}

\begin{proof}
At this point it is useful to recall from \cite[App.~A]{Gil2020} the autoequivalence $\Sigma$ of $\Ho(\A)$. For $X$ in $\A$ we denote by $\Sigma X$ the cokernel of a momomorphism from $X$ to an object in $\W_{\A}$, which is uniquely determined, up to isomorphism, in $\Ho(\A)$. In fact, $\Sigma$ defines an autoequivalence of $\Ho(\A)$.

(i)$\Longrightarrow$(ii)$\colon$Assume that $\textbf{R}r$ is fully faithful. For all $X,Y$ in $\A$ we have an induced bijective map
\begin{equation}
\label{eq:induced_map}
 \xymatrix{
\Hom_{\Ho(\A)}(X,Y)   \ar[r]^-{\phi_{X,Y}}  & \Hom_{\Ho(\B)}(\textbf{R}rX,\textbf{R}rY)= \Hom_{\Ho(\B)}(rRX, rRY) 
  }
\end{equation}
where $RX$ and $RY$ are fibrant replacements of $X$ and $Y$ respectively. In addition, from \cite[Theorem~7.3.(4)]{Gil2020}, for all $n\geqslant 1$, there is an isomorphism $\Hom_{\Ho(\A)}(X, \Sigma^nY)\cong \Ext_{\A}^n(QX, RY)$ and similarly, $\Hom_{\Ho(\B)}(rRX,\Sigma^nrRY)\cong \Ext_{\B}^n(QrRX, rRY)$. We also have that in $\Ho(\B)$: 
\begin{equation}
\label{eq:sigma_commutes}
rR\Sigma^nY=\textbf{R}r\Sigma^nY\cong \Sigma^n \textbf{R}rY=\Sigma^n rRY,
\end{equation}
where the isomorphism in the middle follows since $\mathbf{R}r$ is triangulated by Corollary~\ref{cor:triang_Quillen_adj}.
Hence we have the following commutative diagram for all $n\geqslant 1\colon$
\begin{equation}
\label{comsquareExt}
 \xymatrix{
\Hom_{\Ho(\A)}(X, \Sigma^nY)   \ar[r]^-{\phi_{X,\Sigma^{n}Y}} \ar[d]^{\cong}  & \Hom_{\Ho(\B)}(rRX,rR\Sigma^nY) \ar[d]^{\cong} \\
\Ext_{\A}^n(QX, RY) \ar[r] & \Ext_{\B}^n(QrRX, rRY). 
}
\end{equation}
The upper horizontal map being an isomorphism implies that the lower one is.

(ii)$\Longrightarrow$(i)$\colon$ Assuming that condition (ii) holds, we obtain in particular that, for all $X,Y$ in $\A$, if we choose some $\Sigma X$ in $\A$ (which is uniquely determined in $\Ho(\A)$), there is an isomorphism $\Ext_{\A}^1(Q\Sigma X, RY)\cong \Ext_{\B}^1(QrR\Sigma X, rRY)$, which via the commutative square (\ref{comsquareExt}) (for $n=1$) induces an isomorphism \[\Hom_{\Ho(\A)}(\Sigma X, \Sigma Y) \cong \Hom_{\Ho(\B)}(\mathbf{R}r \Sigma X, \mathbf{R}r \Sigma Y).\] The result now follows since $\Sigma $ is an automorphism of $\Ho(\A)$.

Finally, in case $r$ is fully faithful, the bi-implication (i) $\Longleftrightarrow$ (iii) holds for all model categories and follows easily once we recall, for instance from \cite[Prop.~1.3.13]{hoveybook}, that the counit of the derived adjunction $\mathbf{L}l\colon\Ho(\B)\leftrightarrows\Ho(\A)\colon\mathbf{R}r$ is given, at any $X$ in $\Ho(\A)$, by the composite
\[ 
lQrRX\xrightarrow{lq_{rRX}}lrRX\xrightarrow{\epsilon_{RX}} RX;
\]
here $R$ denotes the fibrant replacement functor in $\A$, $Q$ denotes the cofibrant replacement functor in $\B$ and $\epsilon$ is the ordinary counit of $l\dashv r$, which is an isomorphism since $r$ was assumed fully faithful.
\end{proof}

We also state the dual result for later use, the proof is left to the reader.

\begin{proposition} 
\label{prop:fully_faithful_left_derived}
Let $\A$ and $\B$ be two hereditary abelian model structures with Hovey triples $(\C_{\A},\W_{\A},\F_{\A})$ and $(\C_{\B},\W_{\B},\F_{\B})$ respectively. Let $l\colon\B\leftrightarrows\A\colon r$ be a Quillen adjunction, where $l$ is the left adjoint. The following are equivalent:
\begin{itemize}
\item[(i)] The total left derived functor $\textbf{L}l$ is fully faithful.

\item[(ii)] For all $X,Y\in \B$ and $n\geqslant 1$,\, $\Ext_{\B}^n(RY, QX)\cong \Ext_{\A}^n(lQX, RlQY)$.
\end{itemize}
In case $l$ is fully faithful, the above statements are also equivalent to:
\begin{itemize}
\item[(iii)] For all $X$ cofibrant in $\B$, the weak equivalence $r^{\A}_{lX}\colon lX\rightarrow R_{\A}(lX)$ is mapped to a weak equivalence under $r$.
\end{itemize}
\end{proposition}

The last two results of this section concern Quillen adjoint triples in relation with fully faithful derived functors.

\begin{proposition} 
\label{prop:left_right_derived_embedding}
Let $\mathrm{Q_{tr}}(\A_{1/2},\B_{1/2},\alpha)$ be a Quillen adjoint triple. Then $\mathbf{R}i$ is fully faithful if and only if $\mathbf{L}i'$ is fully faithful. 
\end{proposition}

\begin{proof}
By definition of a Quillen adjoint triple, there is a natural isomorphism $\mathbf{ho}(\alpha)\colon\mathbf{L}g\circ \mathbf{R}i\Rightarrow\mathbf{L}i'\circ\mathbf{L}f$, where the functors $\mathbf{L}f$ and $\mathbf{L}g$ are  (in particular) fully faithful. Thus $\mathbf{L}i'$ is fully faithful if and only if $\mathbf{R}i$ is fully faithful.
\end{proof}

\begin{proposition}
\label{prop:left_right_fully_faithful}
Let $\mathrm{Q_{tr}}(\A_{1/2},\B_{1/2},\alpha)$ be a Quillen adjoint triple as in Definition \ref{def:Quillen_adjoint_triple}, where $f, f', g, g'$ are identity morphisms, $i=i'$ and $\alpha=\mathrm{id}$. Assume that the functors $p$ and $q$ are fully faithful and that $\B_1$ and $\B_2$ share the same class of weak equivalences. Then the functors $\textbf{L}q$ and $\textbf{R}p$ are fully faithful.
\end{proposition}

\begin{proof}
It suffices to prove that $\textbf{R}p$ is fully faithful since in that case $\textbf{L}q$ will be fully faithful too, by the adjoint triple (\ref{eq:adjoint_triple_ab2}),  or (\ref{eq:adjoint_triple_mixed2}). According to Proposition \ref{prop:fully_faithful_right_derived}, we need to prove that for any fibrant object $X$ in $\B_{2}$, the weak equivalence (trivial fibration) $Q_{\A_2}(pX)\xrightarrow{q^{\A_2}_{pX}} p(X)$ in $\A_2$ is mapped to a weak equivalence \[(i(Q_{\A_2}(pX))\xrightarrow{i(q^{\A_2}_{pX})} i (p(X))\xrightarrow{\epsilon_{X}}X\]
in $\B_{2}$ (here the counit map $\epsilon_{X}$ is an isomorphism since $p$ is fully faithful). 
Since $\mathrm{id}\colon \A_2\rightarrow\A_1$ is right Quillen, $q^{\A_2}_{pX}$ is a trivial fibration in $\A_1$. Since $i\colon \A_1\rightarrow \B_1$ is also a right Quillen functor, it will map the trivial fibration $q^{\A_2}_{pX}$ in $\A_1$ to a trivial fibration $i\circ q^{\A_2}_{pX}$ in $\B_1$. But $\B_1$ and $\B_2$ share the same weak equivalences, thus $i\circ q^{\A_2}_{pX}$ is a weak equivalence in $\B_2$ as needed.
\end{proof}

\section{A double pseudofunctor}
\label{sec:a_double_psadofunctor}
The concept of Quillen adjoint triples, as in Definition~\ref{def:Quillen_adjoint_triple}, is related to comparing composites of left and right Quillen functors. Shulman \cite{Shulman} has studied such situations under the prism of double categories. In this section we recall the concept of double categories and relate them with Quillen adjoint triples in Theorem \ref{thm:double_vs_triple}.
In most of the section we follow closely Shulman \cite{Shulman}.

\begin{definition}
A \textbf{double category} $\underline{C}$ consists of the following data.
Two categories, the \textit{horizontal} $\mathcal{H}$ and the \textit{vertical} $\mathcal{V}$ which consist of the same objects (or 0-cells),
two types of morphisms (or 1-cells), the morphisms of $\mathcal{H}$ which are called \textit{horizontal}, and the morphisms of $\mathcal{V}$ which are called \textit{vertical},
and squares (or 2-cells) of the form 
\begin{equation}
\label{eq:2-cell}
  \xymatrix@C=2pc{
a\ar[r]^-{h} \drtwocell\omit{\alpha}  \ar[d]_-{v} & b\ar[d]^-{v'} \\
c \ar[r]_{h'}    & d
}\!
\end{equation}
where $a,b,c,d$ are objects, $h,h'$ are horizontal morphisms and $v,v'$ are vertical morphisms. This data is subject to the following conditions.
For each object $a$ we denote by $1^a$ its identity in $\mathcal{H}$ and by $1_a$ its identity in $\mathcal{V}$ and we require that for every arrow $v\colon a\rightarrow b$ in $\mathcal{V}$ there exists an \textit{identity 2-cell}  
\begin{equation}
  \xymatrix@C=2pc{
\ar[r]^-{1^a} \drtwocell\omit{1^v}  \ar[d]_-{v} & \ar[d]^-{v} \\
 \ar[r]_{1^b}    & 
  }\!
  \nonumber
  \end{equation}
and similarly, we require that for every arrow $h\colon a\rightarrow c$ in $\mathcal{H}$ there exists an \textit{identity 2-cell}
\begin{equation}
  \xymatrix@C=2pc{
\ar[r]^-{h} \drtwocell\omit{1_h}  \ar[d]_-{1_a} & \ar[d]^-{1_c} \\
 \ar[r]_{h}    & \,\,\,\,\,\, . 
  }\!
 \nonumber
  \end{equation}
We require that the 2-cells can be composed horizontally and vertically, forming categories in each direction. 
We denote horizontal composition of 2-cells by $\alpha\boxbar\beta$, or diagramatically by
\begin{equation}
  \xymatrix@C=2pc{
\ar[r] \drtwocell\omit{\alpha}  \ar[d] & \ar[d] \ar[r]  \drtwocell\omit{\beta} & \ar[d]  \\
   \ar[r]& \ar[r] & \,\,\, .
}\!
\nonumber
\end{equation}
  
Similarly, vertical composition of 2-cells is denoted by $\beta'\boxminus\alpha'$ or 
\begin{equation}
\xymatrix@C=2pc{
\ar[r] \drtwocell\omit{\alpha'}  \ar[d] & \ar[d] \\
\ar[r]  \drtwocell\omit{\beta'} \ar[d] & \ar[d] \\
\ar[r] & 
}\!
\nonumber
\end{equation}
and the following interchange law is required to hold
\[(\alpha\boxbar\beta)\boxminus(\gamma\boxbar\delta)=(\alpha\boxminus\gamma)\boxbar(\beta\boxminus\delta).\]

Finally, the following laws are required to hold: $1^{1_a}=1_{1^a}$, $1^{v'}\boxminus 1^{v}=1^{v'v}$ and $1_h\boxbar 1_{h'}=1_{h'h}$.
\end{definition}

The following examples are relevant to this paper.

\begin{example}
\label{ex:double_cats}
\begin{itemize}
\item[(i)]The double category $\underline{\mathit{Cat}}$ where the objects are categories, the 1-cells are functors (horizontal and vertical) and  the 2-cells as in (\ref{eq:2-cell}) are natural transformations $\alpha\colon v'\circ h\Rightarrow h'\circ v$.
\item[(ii)] The double category $\underline{\mathit{Trian}}\mathrm{g}$ where the objects are triangulated categories, the 1-cells are triangulated functors (horizontal and vertical) and  the 2-cells as in (\ref{eq:2-cell}) are natural transformations $\alpha\colon v'\circ h\Rightarrow h'\circ v$.
\item[(iii)] The double category $\underline{\mathit{Model}}$ where the objects are model categories, the horizontal 1-cells are right Quillen functors, the vertical 1-cells are left Quillen functors, and the 2-cells as in (\ref{eq:2-cell}) are natural transformations $\alpha\colon v'\circ h\Rightarrow h'\circ v$.
\item[(iv)] The double category $\underline{\mathit{hAbModel}}$ where the objects are hereditary abelian model categories, the horizontal 1-cells are right Quillen functors, the vertical 1-cells are left Quillen functors, and the 2-cells as in (\ref{eq:2-cell}) are natural transformations $\alpha\colon v'\circ h\Rightarrow h'\circ v$.
 \end{itemize}
\end{example}

An appropriate notion of morphism between double categories is that of a \textit{double pseudofunctor} as in \cite[Definition~6.1]{Shulman}. Part of the definition of a double pseudofunctor is that certain coherence axioms in both the vertical and horizontal directions need to hold. We do not wish to repeat the definition here thus we refer to \textit{loc.cit} for the technical details.

Shulman in \cite[Theorem.~7.6]{Shulman} constructs a double pseudofunctor 
\begin{equation}
\mathbf{Ho}\colon\underline{\mathit{Model}}\rightarrow\underline{\mathit{Cat}}
\nonumber
\end{equation}
which maps a model category to its homotopy category, a right Quillen functor to its right derived functor, a left Quillen functor to its left derived functor, and a 2-cell
\begin{equation}
  \xymatrix@C=3pc{
\A\ar[r]^-{h} \drtwocell\omit{\alpha}  \ar[d]_-{v} & \C\ar[d]^-{v'} \\
\B \ar[r]_{h'}    & \D
  }\!
  \nonumber
  \end{equation}
to the following 2-cell
\begin{equation}
  \xymatrix@C=3pc{
\Ho(\A)\ar[r]^-{\textbf{R}h} \drtwocell\omit{\,\,\,\,\,\,\,\,\,\,\,\,\mathbf{ho}(\alpha)}  \ar[d]_-{\textbf{L}v} & \Ho(\C)\ar[d]^-{\textbf{L}v'} \\
\Ho(\B) \ar[r]_{\textbf{R}h'}    & \Ho(\D)
  }\!
  \nonumber
\end{equation}
defined by the following composite of zigzags in $\Ho(\D)$, where $R$ and $Q$ refer to fibrant and cofibrant replacement functors to be interpreted in the appropriate category:
\begin{equation}
\label{eq:composite_map_triples_Shulman}
\xymatrix@C=2pc{
v'QhR  & 
v'QhQR \ar[l]^-{\sim} \ar[r]^-{v'q^{\C}_{h QR}} &
v'hQR \ar[d]^-{\alpha_{QR}} & &   \\
& & 
h'v QR \ar[r]^-{h'R^{\B}_{v QR}}  &
h'RvQR &
h'RvQ. \ar[l]^-{\sim}
\nonumber
}\!
\end{equation}

Proposition \ref{prop:left_Quillen_triang} together with \cite[Theorem~7.6]{Shulman} combine to give the following observation.

\begin{proposition}
The double pseudofunctor $\mathbf{Ho}\colon\underline{\mathit{Model}}\rightarrow\underline{\mathit{Cat}}$ restricts to a double pseudofunctor
\begin{equation}
\mathbf{Ho}\colon\underline{\mathit{hAbModel}}\rightarrow\underline{\mathit{Trian}}\mathrm{g}
\nonumber
\end{equation}
between the double category of hereditary abelian model structures and the double category of triangulated categories, as defined in Example \ref{ex:double_cats}.
\end{proposition}

\subsection{Relation to Quillen adjoint triples}
\label{subsec:relation_to_Quillen}

We want to interpret Definition \ref{def:Quillen_adjoint_triple} in the context of double categories. Assume that we are given a Quillen adjoint triple $\mathrm{Q_{tr}}(\A_{1/2},\B_{1/2},\alpha)$.
The natural transformation $\alpha$, in the double category $\underline{\mathit{Model}}$, may be written as a 2-cell
\begin{equation}
\label{eq:alpha_in_double_cat}
\xymatrix@C=3pc{
\A_1\ar[r]^-{i} \ddrtwocell\omit{\alpha}  \ar[d]_-{f} & \B_1 \ar[dd]^-{g} \\
\A_2   \ar[d]_{\mathrm{i'}} & \\
\B_2 \ar[r]^-{\mathrm{id}} & \B_2.
}\!
\nonumber
\end{equation}

We remark that the asymmetry of this 2-cell compared to the diagram (\ref{eq:alpha_in_def_of_Quillen_adj}) in Definition \ref{def:Quillen_adjoint_triple} occurs since we need to distinguish between left and right Quillen functors in the double category $\underline{\mathit{Model}}$.

The following result relates the double pseudofunctor $\mathbf{Ho}\colon\underline{\mathit{Model}}\rightarrow\underline{\mathit{Cat}}$ to the concept of a Quillen adjoint triple.

\begin{theorem}
\label{thm:double_vs_triple}
The double pseudofunctor $\mathbf{Ho}\colon\underline{\mathit{Model}}\rightarrow\underline{\mathit{Cat}}$ maps the 2-cell $\alpha$ to a natural isomorphism $\mathbf{Ho}(\alpha)$ if and only if $\mathrm{Q_{tr}}(\A_{1/2},\B_{1/2},\alpha)$ is a Quillen adjoint triple in the sense of Definition \ref{def:Quillen_adjoint_triple}.
\end{theorem}

\begin{proof}
The double pseudofunctor $\mathbf{Ho}$ maps $\alpha$ to the following 2-cell in $\underline{\mathit{Cat}}$ 
\begin{equation}
\label{eq:Hoalpha_in_double_cat}
\xymatrix@C=3pc{
\Ho(\A_1)\ar[r]^-{\textbf{R}i} \ddrtwocell\omit{\,\,\,\,\,\,\,\,\,\,\,\,\mathbf{Ho}(\alpha)}  \ar[d]_-{\textbf{L}f} & \Ho(\B_1) \ar[dd]^-{\textbf{L}g} \\
\Ho(\A_2)   \ar[d]_{\textbf{L}\mathrm{i'}} & \\
\Ho(\B_2) \ar[r]^-{\textbf{R}\mathrm{id}} & \Ho(\B_2)
}
\nonumber
\end{equation}
which by definition is represented by the following composite of zig-zags, where $v:=i'\circ f$,
\begin{equation}
\label{eq:nat_map_cf_Shulman}
\xymatrix@C=3pc{
gQiR  & 
gQiQR \ar[l]^-{\sim} \ar[r]^-{gq^{\B_1}_{i QR}} &
(g\circ i)QR \ar[d]^-{\alpha_{QR}} & &   \\
& & 
(\mathrm{id}\circ v) QR \ar[r]^-{\mathrm{id}\circ R^{\B_2}_{v QR}}_-{\sim}  &
\mathrm{id}(RvQR) &
\mathrm{id}(RvQ). \ar[l]^-{\sim}
\nonumber
}\!
\end{equation}
Note that the map $\mathrm{id}\circ R^{\B_2}_{v QR}$ is a weak equivalence (we consider a fibrant replacement in $\B_2$ which will remain a weak equivalence after applying $\mathrm{id}\colon \B_2\rightarrow \B_2$). Hence modulo weak equivalences the above composite is represented by 
\begin{equation}
\label{eq:nat_map_cf_Shulman_2}
\xymatrix@C=3pc{
gQiQR \ar[r]^-{gq^{\B_1}_{i QR}} & (g\circ i)QR \ar[r]^-{\alpha_{QR}}  & (i'\circ f) QR 
\nonumber
}\!
\end{equation}
and exactly the same is true for the composite (\ref{eq:composite_map_triples}) in Definition \ref{def:Quillen_adjoint_triple}. 
\end{proof}

\begin{corollary}
\label{cor:Shulman}
The composite of zig-zags (\ref{eq:composite_map_triples}) in Definition \ref{def:Quillen_adjoint_triple} is a natural isomorphism in $\Ho(\B_2)$ if and only if for any fibrant-cofibrant object $X$ in $\A_1$, the composite
\begin{equation}
  \xymatrix@C=2pc{
gQ_{\B_{1}}(iX) \ar[r]^-{gq^{\B_{1}}_{iX}} & (g\circ i)X \ar[r]^-{\alpha_{X}} & (i'\circ f)X
}\!
\nonumber
\end{equation}
is a natural isomorphism in $\Ho(\B_{2})$.
\end{corollary}

\begin{proof}
From Remark \ref{rem:fibrant-cofibrant objects} we have that for any fibrant-cofibrant object $X$ in $\A_1$, the composite of zig-zags (\ref{eq:composite_map_triples}) in Definition \ref{def:Quillen_adjoint_triple} reduces to the composite displayed, which represents the isomorphism $\mathbf{ho}(\alpha)_{X}$. From Theorem \ref{thm:double_vs_triple} we have that $\mathbf{Ho}(\alpha)_{X}$ is an isomorphism too. Now we may apply \cite[Rem.~7.2]{Shulman}, which says that for any object $Y$ in $\A_1$,\, $\mathbf{Ho}(\alpha)_{Y}$ represents an isomorphism in $\Ho(\B_2)$ and apply once more Theorem \ref{thm:double_vs_triple} to deduce that $\mathbf{ho}(\alpha)_{Y}$ represents an isomorphism in $\Ho(\B_2)$.
\end{proof}

\begin{remark}
A nice aspect of the formalism of double categories given in this section and of the double pseudofunctor $\mathbf{Ho}\colon\underline{\mathit{Model}}\rightarrow\underline{\mathit{Cat}}$ is that we gain insight in composing Quillen adjoint triples or possibly more complicated structures, like Quillen quadruples, higher ladders of Quillen adjoints and so on. For instance, if $\mathrm{Q_{tr}}(\A_{1/2},\B_{1/2},\alpha\colon g\circ i\Rightarrow i'\circ f)$ and $\mathrm{Q_{tr}}(\B_{1/2},\C_{1/2},\beta\colon h\circ e\Rightarrow e'\circ g)$ are Quillen adjoint triples, we may splice together the 2-cells defined by $\alpha$ and $\beta$ to obtain a 2-cell $\gamma:=(1^e\boxminus \alpha)\boxbar \beta$, as indicated below:
\begin{equation}
\label{eq:alpha_in_double_cat}
\xymatrix@C=3pc{
\A_1\ar[r]^-{i}  \ar[d]_-{f}  \ddrtwocell\omit{\alpha} & \B_1 \ar[dd]^-{g} \ar[r]^-{e} & \C_1 \ar[ddd]^-{h}\\
\A_2   \ar[d]_{\mathrm{i'}}   &  \ddrtwocell\omit{\beta}  &  \\
\B_2 \ar[r]^-{\mathrm{id}}  \drtwocell\omit{1^{e'}} \ar[d]_-{e'} & \B_2 \ar[d]^-{e'} & \\
\C_2 \ar[r]_-{\mathrm{id}} & \C_2 \ar[r]_-{\mathrm{id}} & \C_2.
}\!
\end{equation}
We claim that this defines a Quillen adjoint triple $\mathrm{Q_{tr}}(\A_{1/2},\C_{1/2},\gamma)$. By Theorem \ref{thm:double_vs_triple} we have that $\Ho(\alpha)$ and $\Ho(\beta)$ are natural isomorphisms, thus it suffices to argue that the 2-cell defined by $\Ho(1^{e'})$ is also a natural isomorphism. For this it suffices to check that, for all fibrant-cofibrant objects $X$ in $\B_2$, the following composite is an isomorphism in $\Ho(\C_2)$, 
\[ e'Q_{\B_2}\mathrm{id}X \xrightarrow{e'q^{\B_2}_{\mathrm{id}X}} (e'\circ \mathrm{id})X\xrightarrow{1^{e'}_{X}}(\mathrm{id\circ e'})X\xrightarrow{\mathrm{id}(r_{e'X}^{\C_2})}\mathrm{id}R_{\C_2}e'X.\]
Since $X$ is fibrant-cofibrant in $\B_2$ it is easily checked that this composite is in fact the fibrant replacement $e'X\xrightarrow{r^{\C_2}_{e'X}}R_{\C_2}e'X$ of $e'X$ in $\C_2$, thus a weak equivalence.
\end{remark}

\section{Transfer of abelian model structures}
\label{sec:transfer}

In this section we study transfer (also called lifting) of abelian model structures along a left or a right adjoint. 
The terminology we use in this section is adapted from the general theory of model structures, in connection with right/left lifting of weak factorization systems and model structures, see for instance \cite{lifting}. Here, in view of Fact \ref{fact:wfs_vs_cot_pairs}, we adapt the discussion to cotorsion pairs. We prove a result on right-lifting in Theorem~\ref{thm:right_lifting} and one on left-lifting in Theorem~\ref{thm:left_lifting}.

We will make use of the following setup, which will be specialized further in the main results of this section.

\begin{ipg}\textbf{Setup}
\label{setup_lifting}
We consider a bicomplete abelian category $\M$ equipped with a Hovey triple $(\C_{\M},\W_{\M},\F_{\M})$, two bicomplete abelian categories $\D, \E$, and adjoint pairs $(q, i)$ and $(j, p)$ as follows$\colon$
\begin{equation}
\label{eq:transfer_setup}
\xymatrix@C=1pc{
\D \ar@/_1.2pc/[rr]_-{i} & \bot & \M\ar@/_1.2pc/[ll]_-{q}
} 
\\ \,\,\,\,\,\,\,\,\,\,\,\mbox{and}\,\,\,\,\,\,\,\,\,\,\,
\xymatrix@C=1pc{
\E \ar@/_1.2pc/[rr]_-{j} & \top & \M\ar@/_1.2pc/[ll]_-{p}
} 
\end{equation}
where the functors $i$ and $j$ are faithful and exact. 

We define the following classes of objects in $\D$:
\[\F_{\D}:=i^{-1}(\F_{\M}),\,\,\,\,\,\, \W_{\D}:=i^{-1}(\W_{\M}),\,\,\,\,\,\,\mbox{and}\,\,\,\,\,\, \C_{\D}:=\leftperp{(\F_{\D}\cap\W_{\D})},\]
which are the candidates for the fibrant, trivial, and cofibrant objects in $\D$ respectively.
We also define the following classes of morphisms in $\D$:
\[
\mathrm{fib}_{\D}:=i^{-1}(\mathrm{fib}_{\M}),\,\,\,\,\,\, \mathrm{weak}_{\D}:=i^{-1}(\mathrm{weak}_{\M})\,\,\,\,\,\,\mbox{and}
\]
\[
\mathrm{cof}_{\D}:={}^{\lifts}\left( \mathrm{fib}_{\D}\cap\mathrm{weak}_{\D}\right),
\]
which are candidates for the fibrations, weak equivalences and cofibrations in $\D$ respectively. 
Similarly, we define the following classes of objects in $\E$:
\[C_{\E}:=j^{-1}(\C_{\M}),\,\,\,\,\,\, \W_{\E}:=j^{-1}(\W_{\M})\,\,\,\,\,\,\mbox{and}\,\,\,\,\,\, \F_{\E}:=\rightperp{(\C_{\E}\cap\W_{\E})}.\]
and the following classes or morphisms in $\E$:
\[\mathrm{cof}_{\E}:=j^{-1}(\mathrm{cof}_{\M}),\,\,\,\,\,\, \mathrm{weak}_{\E}:=j^{-1}(\mathrm{weak}_{\M})\,\,\,\,\,\,\mbox{and}\]
\[\mathrm{fib}_{\E}:=\left( \mathrm{cof}_{\E}\cap \mathrm{weak}_{\E} \right)^{\square}.\]
\end{ipg}

Under Setup \ref{setup_lifting} we define the following notion.

\begin{definition}
\label{dfn:transfers}
We say that the \textbf{right-lifted abelian model structure} (along the right adjoint $i$) exists on $\D$ if $(\C_{\D},\W_{\D},\F_{\D})$ constitutes a Hovey triple on $\D$. Similarly, we say that the \textbf{left-lifted abelian model structure} (along the left adjoint $j$) exists on $\E$ if $(\C_{\E},\W_{\E},\F_{\E})$ constitutes a Hovey triple on $\E$.
\end{definition}

\begin{remark}
Lifting of model structures has been well-studied in the literature. In the case of combinatorial model structures (i.e. locally presentable and cofibrantly generated) the right-lifting goes back to Quillen, while the left-lifting is contained in work of Makkai-Rocisk\'y \cite{cellular}. More general results for accessible model structures were obtained in \cite{lifting}. For the results in this paper, we need the lifted model structures to be abelian (when they exist). To build the analogy with the aforementioned results on combinatorial model structures, we focus on cotorsion pairs generated by sets (since for a locally presentable abelian category, an abelian model structure is combinatorial if and only if its associated complete cotorsion pairs are each generated by a set \cite[Lemma~3.7]{SP}). 
\end{remark}

\begin{lemma}
\label{lemma_trivial_fibrations}
Keeping the same notation as in Setup \ref{setup_lifting}, we have the following:
\item[(i)] The class $\mathrm{fib}_{\D}$ consists precisely of the epimorphisms in $\D$ with kernel in $\F_{\D}$.
\item[(i')] The class $\mathrm{fib}_{\D}\cap\mathrm{weak}_{\D}$ consists precisely of the epimorphisms in $\D$ with kernel in $\F_{\D}\cap\W_{\D}$.
\item[(ii)] $D\in\C_{\D}$  if and only if $0\rightarrow D$ is in $(\mathrm{cof})_{\D}$.
\item[(ii')] $D\in\C_{\D}\cap\W_{\D}$ if and only if $0\rightarrow D$ is in $(\mathrm{cof})_{\D}\cap(\mathrm{weak})_{\D}$.
\item[(iii)] $\F_{\D}\subseteq \rightperp{q(\C_{\M}\cap\W_{\M})}$ and $\F_{\D}\cap\W_{\D}\subseteq \rightperp{q(\C_{\M})}$.
\item[(iv)] If $q$ maps cofibrations to monomorphisms in $\D$ then we obtain equalities in (iii).
\end{lemma}

\begin{proof}
To prove (i) observe that if $f$ in $\D$ is a map such that $i(f)$ is an epi with kernel in $\F_{\M}$, then $f$ is an epi (since $i$ is faithful) and $\ker(i(f))=i\ker(f)$ (since $i$ is a right adjoint), thus $\ker(f)\in\F_{\D}$. Conversely, any epimorphism with kernel in $\F_{\D}$ is mapped to an epimorphism with kernel in $\F_{\M}$ since $i$ is exact. The proof of (i') is identical.

To prove (ii), let $D$ be in $\C_{\D}=\leftperp{(\F_{\D}\cap\W_{\D})}$. 
By definition, we have that $0\rightarrow D$ is in $(\mathrm{cof})_{\D}$ if and only if it has the left-lifting property with respect to morphisms in $(\mathrm{fib})_{\D}\cap(\mathrm{weak})_{\D}$, which from (i') are epimorphisms with kernel in $\F_{\D}\cap\W_{\D}$. But if $f\colon X\rightarrow Y$ is an epimorphism in $\D$ with $\ker(f)\in \F_{\D}\cap\W_{\D}$, the existence of the desirable dotted arrow
\begin{displaymath}
 \xymatrix@C=2pc{
 0 \ar[d] \ar[r] & X \ar@{->>}^-{f}[d] \\
  D \ar@{-->}[ur]\ar[r] & Y
}
\end{displaymath}
follows since $\Ext_{\D}^{1}(D,\ker(f))=0$. Conversely, assume that $0\rightarrow D$ is in $(\mathrm{cof})_{\D}$ and let $K\rightarrowtail Y\twoheadrightarrow D$ be a short exact sequence in $\D$ with $K\in\F_{\D}\cap\W_{\D}$. Then there exists a dotted map 
\begin{displaymath}
 \xymatrix@C=2pc{
 0 \ar[d] \ar[r] & Y \ar@{->>}[d] \\
  D \ar@{-->}[ur]\ar@{=}[r] & D;
  }
\end{displaymath}
which implies that $D\in\leftperp{(\F_{\D}\cap\W_{\D})}$. Similarly on can prove (ii').

(iii) Let $X\in\F_{\D}$ and consider a short exact sequence $0\rightarrow X\rightarrow Z\xrightarrow{\epsilon} q(C)\rightarrow 0$ for some $C\in\C_{\M}\cap\W_{\M}$. Note that by (i) the exact functor $i$ maps $\epsilon$ to a fibration in $\M$. We claim that the short exact sequence $0\rightarrow X\rightarrow Z\xrightarrow{\epsilon} q(C)\rightarrow 0$ splits if and only if there exists a dotted map $\delta$ making the following diagram commutative,
 \begin{displaymath}
 \xymatrix@C=2pc{
 0 \ar[d] \ar[r] & i(Z) \ar[d]^-{i(\epsilon)} \\
  C \ar@{-->}[ur]^-{\delta}\ar[r]^-{\eta_{C}} & iq(C),
}
\end{displaymath}
where $\eta_{C}$ is the unit map. Assume that the short exact sequence $0\rightarrow X\rightarrow Z\xrightarrow{\epsilon} q(C)\rightarrow 0$ splits and denote by $\zeta\colon q(C)\to Z$ a section of $\epsilon$. Then the desired $\delta$ is the composition $i(\zeta)\circ \eta_C$. Conversely, assume that we have a morphism $\delta$ making the above square commutative. Applying the functor $q$ we obtain the following commutative diagram
\[
 \xymatrix@C=2pc{
 0 \ar[d] \ar[r] & qi(Z) \ar[d]^-{qi(\epsilon)} \ar[r]^{\varepsilon_Z} & Z \ar[d]^{\epsilon} \\
  q(C) \ar[ur]^-{q(\delta)}\ar[r]^-{q(\eta_{C})} & qiq(C) \ar[r]^{\varepsilon_{q(C)}} & q(C),
}
\]
where $\varepsilon_Z$ is the counit map. Then it follows that $\epsilon\circ (\varepsilon_Z\circ q(\delta)) = \mathrm{id}_{q(C)}$ showing that the short exact sequence $X\rightarrowtail Z\twoheadrightarrow q(C)$ splits.

To finish the proof note that such a $\delta$ exists since $i(\epsilon)$ is a fibration and $0\rightarrow C$ is a trivial cofibration in the model category $\M$. The other inclusion is proved similarly and is left to the reader.

(iv) Let $Y\in \rightperp{q(\C_{\M}\cap\W_{\M})}$. We need to prove that $i(Y)$ belongs to the class $\F_{\M}=\rightperp{(\C_{\M}\cap\W_{\M})}$ in $\M$. Consider a short exact sequence $0\rightarrow i(Y)\xrightarrow{j} X\rightarrow C\rightarrow 0$ with $C\in\C_{\M}\cap\W_{\M}$. This is a trivial cofibration in $\M$, thus, by assumption, we obtain a commutative diagram with exact rows 
 \begin{displaymath}
 \xymatrix@C=2pc{
 qi(Y) \, \ar[d]_-{\varepsilon_{Y}} \ar@{>->}[r]^-{q(j)}\, & q(X) \ar@{-->}[dl]^-{\delta} \ar[d] \ar@{->>}[r] & q(C) \\
  Y \ar[r] & 0 & 
}
\end{displaymath}
where the map $\delta$ exists since $Y\in \rightperp{q(\C_{\M}\cap\W_{\M})}$. Applying the functor $i$ to the above diagram, we obtain by adjunction (same argument as in (iii)) that $j$ is a split monomorphism. 
The other equality is proved similarly. 
\end{proof}

Our next goal is to give necessary and sufficient conditions for the existence of right/left transferred abelian model structures. We have the following general result on Hovey triples.

\begin{proposition}
\label{prop_retract_argument}
Consider a $($bicomplete$)$\footnote{Bicompleteness is not really needed here but we include it since Hovey triples are defined on bicomplete abelian categories, as in Definition \ref{def:Hovey_triples}.} abelian category $\class D$ with three classes of objects $\C_{D},\W_{\D}$ and $\F_{\D}$, where $\W_{\D}$ is thick and closed under summands. Then  $(\C_{\D},\W_{\D},\F_{\D})$ is a Hovey triple on $\D$ if and only if the following hold: 
\begin{itemize}
\item[(i)] The pairs $(\C_{\D},\W_{\D}\cap\F_{\D})$ and $(\leftperp{\F_{\D}},\F_{\D})$ are complete cotorsion pairs in $\D$.
\item[(ii)] The inclusion $\leftperp{\F_{\D}}\subseteq \W_{\D}$ holds.
\end{itemize}
\end{proposition}

\begin{proof}
If $(\C_{\D},\W_{\D},\F_{\D})$ is a Hovey triple on $\D$, then conditions (i) and (ii) follow immediately by Definition~\ref{def:Hovey_triples}. The non-trivial implication is to show that $(\C_{\D},\W_{\D},\F_{\D})$ is a Hovey triple on $\class D$ assuming conditions (i) and (ii). For this it suffices to show that $\C_{\D}\cap\W_{\D}=\leftperp{\F_{\D}}$. The inclusion $\supseteq$ follows from condition (ii) together with the observation that $\F_{\D}\cap\W_{\D}\subseteq\F_{\D}$ implies $\leftperp{\F_{\D}}\subseteq\leftperp{(\F_{\D}\cap\W_{\D})}=\C_{\D}$. For the converse inclusion, let $X\in\C_{\D}\cap\W_{D}$ and consider a short exact sequence $F\rightarrowtail Y\twoheadrightarrow X$ in $\D$, with $F\in\F_{\D}$. We consider the pullback diagram,
\begin{displaymath}
 \xymatrix@C=2pc{
  & Z \ar@{=}[r] \ar@{>->}[d] & Z \ar@{>->}[d] \\
  F\,\, \ar@{>->}[r] \ar@{=}[d] & P \ar@{->>}[r] \ar@{->>}[d] & C\ar@{->>}[d]\\
  F\,\, \ar@{>->}[r] & Y\ar@{->>}[r] & X ,
 }
 \end{displaymath}
where in the rightmost vertical sequence $C\in\leftperp{\F_{\D}}$ and $Z\in\F_{\D}$ (this makes use of the completeness of $(\leftperp{\F_{\D}},\F_{\D})$). Since $\leftperp{\F_{\D}}\subseteq \W_{\D}$, $X$ is in $\W_{\D}$ and the class $\W_{\D}$ is thick we deduce that $Z\in\F_{\D}\cap\W_{\D}$. In particular, the rightmost vertical sequence splits since $(\C_{\D},\W_{\D}\cap\F_{\D})$ is a cotorsion pair. Moreover the middle horizontal sequence splits by construction. We infer that the bottom sequence splits.
\end{proof}

Following \cite{lifting} we call the inclusion $\leftperp{\F_{\D}}\subseteq \W_{\D}$ \textbf{the right acyclicity condition}. For completeness we state the dual result where its proof is left to the reader.

\begin{proposition}
\label{prop_retract_argument_dual}
Consider a $($bicomplete$)$ abelian category $\class E$ with three classes of objects $\C_{\E},\W_{\E}$ and $\F_{\E}$, where $\W_{\E}$ is thick and closed under summands. Then  $(\C_{\E},\W_{\E},\F_{\E})$ is a Hovey triple on $\E$ if and only if the following hold: 
\begin{itemize}
\item[(i)] The pairs $(\C_{\E}\cap\W_{\E}, \F_{\E})$ and $(\C_{\E},\rightperp{\C}_{\E})$ are complete cotorsion pairs in $\E$.
\item[(ii)] The inclusion $\rightperp{\C}_{\E}\subseteq \W_{\E}$ holds.
\end{itemize}
\end{proposition}

The inclusion $\rightperp{\C}_{\E}\subseteq \W_{\E}$ is called \textbf{the left acyclicity condition}. The following definition will be useful.

\begin{definition}
\label{dfn:transfer_of_cot_pairs}
Let $(\X,\Y)$ be a (functorial) complete cotorsion pair in an abelian category $\M$, and assume that we are given adjoint pairs $q\colon\M\leftrightarrows\D\colon i$ and $j\colon\E\leftrightarrows\M\colon p$ between abelian categories, where $i$ (resp., $j$) is a right (resp., left) adjoint. We say that the \textbf{right-lifted} (functorial) complete cotorsion pair exists in $\D$ \textit{if} $(\leftperp({i^{-1}\Y)},i^{-1}\Y)$ is such a cotorsion pair in $\D$. Dually, we say that the \textbf{left-lifted} (functorial) complete cotorsion pair exists in $\E$ \textit{if} $(j^{-1}(\X),\rightperp{j^{-1}(\X)})$ is such a cotosion pair in $\E$.
\end{definition}

The next result relates liftings of model structures to liftings of complete cotorsion pairs. It follows at once from Propositions \ref{prop_retract_argument} and \ref{prop_retract_argument_dual} together with the observation that $\W_{\D}$ and $\W_{\E}$  from Setup \ref{setup_lifting}  are thick and closed under summands (as preimages of $i$ and $j$ respectively). 

\begin{corollary}
\label{cor:lifting_nec_suf}
Under Setup \ref{setup_lifting}, the right-lifted abelian model structure on $\D$ exists if and only if the right-lifted functorial complete cotorsion pairs $(\C_{\D},\W_{\D}\cap\F_{\D})$ and $(\leftperp{\F_{\D}},\F_{\D})$ exist and the right acyclicity condition  $\leftperp{\F_{\D}}\subseteq \W_{\D}$ holds.

Similarly, the left-lifted abelian model structure on $\E$ exists if and only if the left-lifted functorial complete cotorsion pairs $(\C_{\E}\cap\W_{\E},\F_{\E})$ and $(\C_{\E},\rightperp{\C}_{\E})$ exist and the left acyclicity condition $\rightperp{\C}_{\E}\subseteq \W_{\E}$ holds.
\end{corollary}

We point out the following motivational observation:

\begin{lemma}
\label{lem:motiv}
If the right-lifted abelian model structure on $\D$ exists, the adjunction $i\colon \D \rightleftarrows \M\colon q$ from Setup \ref{setup_lifting} is a Quillen adjunction.
\end{lemma}

\begin{proof}
From Fact \ref{fact:wfs_vs_cot_pairs} we obtain two abelian weak factorization systems\\ $(\mathrm{Mono}(\C_{\D}),\mathrm{Epi}(\W_{\D}\cap\F_{\D}))$ and $(\mathrm{Mono}(\C_{\D}\cap\W_{\D}),\mathrm{Epi}(\F_{\D}))$ where  the classes $\mathrm{Epi}(\F_{\D})$  and $\mathrm{Epi}(\W_{\D}\cap\F_{\D})$ are the fibrations and trivial fibrations of the model structure on $\D$. By Lemma \ref{lemma_trivial_fibrations} these are the classes $\mathrm{fib}_{\D}$ and $\mathrm{fib}_{\D}\cap\mathrm{weak}_{\D}$ respectively, which are mapped to fibrations and trivial fibrations respectively in $\M$, by construction. Hence $i\colon \D\rightarrow \M$ is right Quillen.
\end{proof}

Of course the dual of Lemma \ref{lem:motiv} about left-lifting abelian model structures holds and can be proved with similar arguments.

Before we continue with the main results of this section, we need to recall the following useful notion of filtered objects and some related structural results.

\begin{definition}(\cite[Definition~6.1]{GT})
\label{def:continuous}
Let $\A$ be a Grothendieck category and let $\class{X}$ be a class of objects in $\A$. An object $X$ of $\A$ is called $\class{X}$--\textbf{filtered} if there exists a well ordered system $(X_{\alpha}, f_{\beta\alpha}\colon X_{\alpha}\rightarrow X_{\beta} \ | \ \alpha<\beta \leqslant\lambda)$ in $\A$, indexed by an ordinal number $\lambda$, such that the following hold:
\begin{itemize}
\item[(i)] $X_0=0$ and for each limit ordinal $\sigma\leqslant\lambda$, $X_{\sigma}=\lim\limits_{\longrightarrow}\phantom{}_{\alpha<\sigma} X_{\alpha}$.
\item[(ii)] For all $\alpha<\beta\leqslant\lambda$, the morphisms $f_{\beta\alpha}\colon X_{\alpha}\to X_{\beta}$ are monomorphisms.
\item[(iii)] For all $\alpha<\lambda$, $\Coker{f_{\alpha+1,\alpha}}$ lies in $\class{X}$.
\item[(iv)] $X_{\lambda}=X$.
\end{itemize}
We denote the class of the $\X$--filtered objects in $\A$ by $\mathrm{Filt}(\X)$. In case $\X$ is a (small) set of objects in $\A$ such classes are called \textbf{deconstructible}.
\end{definition}

\begin{proposition}\textnormal{(see \cite[Theorems~5.16,~5.17,~5.18)]{Stoviceksurvey}}
\label{prop:deconstruct}
Let $\A$ be a Grothendieck category. Then the following hold:
\begin{itemize}
\item[(i)] If $\class X$ is a (small) set of objects in $\A$ such that $\mathrm{Filt}(\class X)$ contains a generator, then $(\leftperp{(\rightperp{\class X})},\rightperp{\class X})$ is a complete cotorsion pair in $\A$, and the class $\leftperp{(\rightperp{\class X})}$ consists of summands of direct sums of objects from the class $\Filt(\class X)$.

\item[(ii)] If $(\X,\Y)$ is a complete cotorsion pair which is generated by a set of objects in $\A$, then the class $\X$ is deconstructible (i.e., a class of the form $\Filt(\class S)$ where $\class{S}$ is some set of objects in $\A$), it is closed under retracts and contains a generator; in fact we have $(\X,\Y)=(\Filt(\class S),\rightperp{\Filt(\class S)})=(\leftperp{(\rightperp{\class S})},\rightperp{\class S})$.
\end{itemize}
In particular, for any deconstructible class $\Filt(\class S)$ in $\A$ which is closed under retracts and contains a generator, the pair $(\Filt(\class S),\rightperp{\Filt(\class S)})$ is a complete cotorsion pair in $\A$.
\end{proposition}

We can now state and prove the main results of this section.

\begin{theorem}\textnormal{(Right-lifting of abelian model structures)}
\label{thm:right_lifting}
Under Setup \ref{setup_lifting}, assume in addition that the categories $\D$ and $\M$ are Grothendieck and that the following hold:
\begin{itemize}
\item[(i)] $q$ maps cofibrations in $\M$ to monomorphisms in $\D$.
\item[(ii)] The pairs $(\C_{\M}\cap\W_{\M},\F_{\M})$ and $(\C_{\M},\W_{\M}\cap\F_{\M})$ are complete cotorsion pairs which are each generated by a set.
\item[(iii)] The right acyclicity condition $\leftperp{\F_{\D}}\subseteq \W_{\D}$ holds.
\end{itemize}
Then the right-lifted abelian model structure $(\C_{\D},\W_{\D},\F_{\D})$ exists on $\D$ and is hereditary if the given Hovey triple on $\M$ is.
\end{theorem}

\begin{proof}
From Corollary \ref{cor:lifting_nec_suf}, it suffices to show that the pairs $(\C_{\D},\W_{\D}\cap\F_{\D})$ and $(\leftperp{\F_{\D}},\F_{\D})$ are complete cotorsion pairs in $\D$. By condition (ii) and Proposition \ref{prop:deconstruct} (ii) we may assume that $(\C_{\M}\cap\W_{\M},\F_{\M})=(\leftperp{(\rightperp{\class S})},\rightperp{\class S})=(\mathrm{Filt}(\class S),\rightperp{\mathrm{Filt}(\class S)})$ for a set $\mathcal{S}$, such that $\mathrm{Filt}(\class S)$ is closed under retracts and contains a generator. We claim that there is a complete cotorsion pair 
\begin{equation}
\label{eq:cot_pairs}
(\leftperp{\F_{\D}},\F_{\D})=(\leftperp{(\rightperp{q(\class S)})},\rightperp{q(\class S)})
\end{equation}
in $\D$. One can show, using an identical argument as the one given in Lemma~\ref{lemma_trivial_fibrations} (iv), that $\F_{\D}=\rightperp{q(\class S)}$; hence we obtain the equality in (\ref{eq:cot_pairs}). To prove completeness it suffices to argue that the class $\mathrm{Filt}(q(\class S))$ contains a generator (because then we can apply Proposition \ref{prop:deconstruct} (i)). 
For this let $G\in\mathrm{Filt}(\mathcal{S})$ be a generator of $\M$ and note that, since $i$ is faithful, $q(G)$ is a generator of $\class D$ (by adjunction). Moreover, since by assumption, $q$ maps monomorphisms with cokernel in $\class S$ (which in particular are cofibrations in $\M$) to monomorphisms in $\A$, and also commutes with colimits, we obtain that $q(G)\in q\mathrm{Filt}(\mathcal{S})\subseteq\mathrm{Filt}(q(\mathcal{S}))$.  

Using very similar arguments one can prove the existence of a complete cotorsion pair in $\D$:
\begin{equation}
\label{eq:cot_pairs_2}
(\C_{\D},\W_{\D}\cap\F_{\D})=(\leftperp{(\rightperp{q(\class S')})},\rightperp{q(\class S')}),
\end{equation}
for a set $\class S'$ such that $(\C_{\M},\W_{\M}\cap\F_{\M})=(\leftperp{(\rightperp{\class S'})},\rightperp{\class S'})=(\mathrm{Filt}(\class S'),\rightperp{\mathrm{Filt}(\class S')})$. 
Indeed, via Lemma~\ref{lemma_trivial_fibrations} (iv), and the definition of $\C_{\D}:=\leftperp{(\W_{\D}\cap\F_{\D})}$, we obtain the equality in (\ref{eq:cot_pairs_2}). Then one shows that $\mathrm{Filt}(q(\mathcal{S}'))$ contains a generator of $\D$ and employs Proposition \ref{prop:deconstruct} (i).

To prove the final assertion, assume that the Hovey triple $(\C_{\M},\W_{\M},\F_{\M})$ is hereditary. 
Since the functor $i$ is exact, it follows that the classes $\F_{\D}:=i^{-1}(\F_{\M})$ and $\F_{\D}\cap\W_{\D}:=i^{-1}(\F_{\M}\cap\W_{\M})$ are closed under cokernels of monomorphisms. Then from \cite[Lemma~5.24]{GT} we deduce that  $\Ext^i_{\D}(\C_{\D}\cap \W_{\D}, \F_{\D})=0$ and $\Ext^i_{\D}(\C_{\D}, \W_{\D}\cap \F_{\D})=0$ for all $i\geq 1$.
\end{proof}

\begin{theorem} \textnormal{(Left-lifting of abelian model structures)}
\label{thm:left_lifting}
Under Setup \ref{setup_lifting}, assume in addition that the categories $\E$ and $\M$ are Grothendieck, and that the following hold:
\begin{itemize}
\item[(i)] The functor $j\colon \E\rightarrow \M$ has a left adjoint.
\item[(ii)] The pairs $(\C_{\M}\cap\W_{\M},\F_{\M})$ and $(\C_{\M},\W_{\M}\cap\F_{\M})$ are complete cotorsion pairs which are each generated by a set.
\item[(iii)] The functor $j$ maps a generator of $\E$ to $\C_{\M}\cap\W_{\M}$.
\item[(iv)] The left acyclicity condition $\C_{\E}^{\bot}\subseteq\W_{\E}$ holds.
\end{itemize}
Then the left-lifted abelian model structure $(\C_{\D},\W_{\D},\F_{\D})$ exists on $\D$ and is hereditary if the given Hovey triple on $\M$ is.
\end{theorem}

\begin{proof}
From Corollary \ref{cor:lifting_nec_suf}, it suffices to show that the pairs $(\C_{\E}\cap \W_{\E},\F_{\E})$ and $(\C_{\E},\rightperp{\C_{\E}})$ are complete cotorsion pairs in $\E$. From Proposition \ref{prop:deconstruct} (ii) there exists a set $\class S$ for which $(\C_{\M},\W_{\M}\cap \F_{\M})=(\leftperp{(\rightperp{\class S})},\rightperp{\class S})=(\mathrm{Filt}(\class S),\rightperp{\mathrm{Filt}(\class S)})$, such that $\mathrm{Filt}(\class S)$ is closed under retracts and contains a generator.

In order to employ \cite[Proposition~A.10]{Becker} and deduce that the class $j^{-1}\C_{\M}=\C_{\E}$ is deconstructible, we need the functor $j$ to be monadic. Since $j$ is faithful and has a left and a right adjoint, we deduce, from \cite[Theorem~4.4.4]{hand2} for instance, that $j$ is monadic. 
In addition, the class $j^{-1}\C_{\M}=\C_{\E}$ is closed under retracts (since $\C_{\M}$ is closed under retracts) and by condition (iii), it contains a generator. Hence, by Proposition \ref{prop:deconstruct} we deduce that $(\C_{\E},\rightperp{\C_{\E}})$ is a complete cotorsion pair. 

Similarly, one can prove that $(\C_{\E}\cap \W_{\E},\F_{\E})$ is a complete cotorsion pair. Indeed, from Proposition \ref{prop:deconstruct} (ii) we may assume that $\C_{\M}\cap \W_{\M}$ is deconstructible and by Becker's result \cite[Proposition~A.10]{Becker} we obtain that $j^{-1}(\C_{\M}\cap \W_{\M})=j^{-1}\C_{\M}\cap j^{-1}\W_{\M}=\C_{\E}\cap\W_{\E}$ is deconstructible, which contains a generator by condition (iii); hence the result follows from Proposition \ref{prop:deconstruct}.

The final assertion on hereditariness follows from a dual argument as the one in the proof of Theorem \ref{thm:right_lifting}, namely, if we assume that $(\C_{\M},\W_{\M},\F_{\M})$ is a hereditary Hovey triple, then the classes $\C_{\E}=j^{-1}\C_{\M}$ and $\C_{\E}\cap\W_{\E}=j^{-1}(\C_{\M}\cap\W_{\M})$ are closed under kernels of epimorphisms (since $\C_{\M}$ and $\C_{\M}\cap\W_{\M}$ have this property), hence $(\C_{\E},\W_{\E},\F_{\E})$ is a hereditary Hovey triple from \cite[Lemma~5.24]{GT}.
\end{proof}

We close this section with two examples illustrating that in the context of chain complexes, transferring the (usual) projective and injective model structures induces model structures of the same kind. An example of different flavour is given in Example \ref{ex:Gor_morphism_triple}.

We recall the following folklore facts. For a more detailed recollection the reader may consult \cite[Theorems 6.9 and 8.6]{SP} and the references given therein.

\label{ipg_proj_inj_models}
Let $\B$ be a Grothendieck abelian category with enough projective objects.  There exist the following two model structures on the category of complexes $\Ch(\B)$:
\begin{itemize}
\item[(\textsf{proj})] The projective one, $\Ch(\B)_{\mathsf{proj}}$, which is given by the hereditary Hovey triple 
\[(\leftperp{\Chac(\B)},\Chac(\B),\Ch(\B)).\]
\end{itemize}
The objects in the class $\leftperp{\Chac(\B)}$ are called \textit{semi-projective} (or \textit{dg-projective}) complexes. In fact, the cotorsion pair $(\leftperp{\Chac(\B)},\Chac(\B))$ is generated by the set $\mathcal{S}=\{S^{n}(G)\,|\,n\in\mathbb{Z}\}$, where $G$ is a projective generator of $\B$ and $S^{n}(G)$ denotes the stalk complex which has a copy of $G$ in degree $n$. Note that from Proposition \ref{prop:deconstruct} we deduce that the class of semi-projective complexes in $\Ch(\B)$ coincides with summands from the class $\mathrm{Filt}(\mathcal{S})$. Also, the cotorsion pair $(\mathrm{Proj}\Ch(\B),\Ch(\B))$ is generated by the set of all complexes of the form $D^{n}(G):=(0\rightarrow G=G\rightarrow 0)$ that are concentrated in degrees $n$ and $n-1$; this is a set of projective generators for $\Ch(\B)$.
\begin{itemize}
\item[(\textsf{inj})] The injective one, $\Ch(\B)_{\mathsf{inj}}$, which is given by the hereditary Hovey triple 
\[(\Ch(\B),\Chac(\B),\rightperp{\Chac(\B)}).\]
\end{itemize}
The objects in the class $\rightperp{\Chac(\B)}$ are called \textit{semi-injective} (or \textit{dg-injective}) complexes. The cotorsion pairs $(\Ch(\B),\mathrm{Inj}\Ch(\B))$ and $(\Chac(\B),\rightperp{\Chac(\B)})$ are also generated by sets. As recalled for instance in \cite[Theorem~8.6]{SP}, there are sets $\X$ and $\X'$ of objects in $\Ch(\B)$ such that $(\Chac(\B),\rightperp{\Chac(\B)})=(\mathrm{Filt}(\X),\rightperp{\mathrm{Filt}(\X)})$ and $(\Ch(\B),\mathrm{Inj}\Ch(\B))=(\mathrm{Filt}(\X'),\rightperp{\mathrm{Filt}(\X')})$. Here the classes $\mathrm{Filt}(\X)$ and $\mathrm{Filt}(\X')$ are closed under retracts and contain generators.

Note that both these model structures have the same class of trivial objects $\Chac(\B)$ and the same homotopy category, which is the unbounded derived category $\D(\B)$. Evidently, the identity functor from $\Ch(\B)_{\mathsf{proj}}$ to $\Ch(\B)_{\mathsf{inj}}$ is left Quillen and a Quillen equivalence.  

\begin{remark}
\label{rem:induced_rec_cpx}
Let $\mathsf{R_{\mathsf{ab}}}(\A,\B,\C)$ be a recollement of abelian categories. Then there exists an associated recollement of categories of complexes $\mathsf{R_{\mathsf{ab}}}(\Ch(\A),\Ch(\B),\Ch(\C)),$ that we denote by the same functors:
\begin{equation}
\label{eq:rec_cpx}
  \xymatrix@C=4pc{
\Ch(\A)   \ar@<0.0ex>[r] |-{\, i\,}   &  
  \Ch(\B) \ar@/^1.4pc/[l] ^-{p}    \ar@/_1.4pc/[l] _-{q}  \ar[r] |-{\, e \,}&
  \Ch(\C). \ar@/_1.4pc/[l]  _-{l}  \ar@/^1.4pc/[l] ^-{r}
  }
\end{equation}
\end{remark}

\begin{proposition}
\label{prop:transfer_trivial_proj}
Consider an adjoint pair $q\colon \B\rightleftarrows \A\colon i$ between Grothendieck categories with enough projective objects, where the right adjoint $i$ is faithful and exact. Then the right transfer of $\Ch(\B)_{\mathrm{proj}}$ along $i$ exists and coincides with $\Ch(\A)_{\mathrm{proj}}$.  
\end{proposition}

\begin{proof}
We  check the conditions of Theorem \ref{thm:right_lifting}. Condition (ii) on $\Ch(\B)_{\mathrm{proj}}$ is satisfied as we recalled above. To check condition (iii), note first that since all the objects of $\Ch(\B)$ are fibrant, we have that $\F_{\Ch(\A)}:=i^{-1}\F_{\Ch(\B)}=\Ch(\A)$; hence condition (iii) asks for the projective objects in $\Ch(\A)$ to be mapped to acyclic complexes in $\Ch(\B)$. But this holds since the projective objects in $\Ch(\B)$ are split exact complexes of projectives. 

To check that condition (i) holds, we need $q$ to map a cofibration in $\Ch(\B)$, i.e., a monomorphism $\iota\colon X\rightarrow Y$ with $Z:=\coker(\iota)$ a dg-projective complex, to a monomorphism $q(\iota)\colon qX\rightarrow qY$ in $\Ch(\B)$. That is, for all integers $c$, we need the induced map $q(\iota)_c\colon qX(c)\rightarrow qY(c)$ to be a monomorphism in $\A$. For this we consider the following exact sequence where the left-most term $L_{1}q(Z(c))$ is the first left derived functor of $q$ evaluated at $Z(c)$,
\begin{equation}
L_{1}q(Z(c))\rightarrow qX(c)\rightarrow qY(c)\rightarrow qZ(c)\rightarrow 0.
\nonumber
\end{equation}
Since $Z(c)$ is a projective object in $\B$, $L_{1}q(Z(c))$ vanishes. Hence, we conclude that the right transfer of $\Ch(\B)_{\mathrm{proj}}$ along $i$ exists. In order to prove that it coincides with $\Ch(\A)_{\mathrm{proj}}$, it suffices to show that they share the same trivial objects. By construction of the right-transferred model on $\Ch(\A)$,  its class of trivial objects is $i^{-1}\Ch_{\mathrm{ac}}(\B):=\{X\in\Ch(\A)\, |\, iX\in\Ch_{\mathrm{ac}}(\B)\}$. Since $i$ is exact we clearly have that $\Ch_{\mathrm{ac}}(\A)\subseteq i^{-1}\Ch_{\mathrm{ac}}(\B)$. To obtain the reverse inclusion, we observe that if $X\in\Ch(\A)$ is such that $iX$ has zero homology groups $H_{n}(iX)$ for all integers $n$, then we also have that $iH_{n}(X)=0$; thus $H_{n}(X)=0$ since $i$ is faithful.
\end{proof}

We also have the dual of the above Proposition:
\begin{proposition}
\label{prop:transfer_trivial_inj}
Consider an adjoint triple between Grothendieck categories with enough projective objects,
\begin{equation}
\xymatrix@C=3pc{
\A \ar[r]^-{i} & \B\ar@/_1.2pc/[l]_-{q} \ar@/^1.2pc/[l]_-{p} 
} 
\nonumber
\end{equation}
where the functor $i$ is faithful and exact. Then the left transfer of $\Ch(\B)_{\mathrm{inj}}$ along $i$ exists and coincides with $\Ch(\A)_{\mathrm{inj}}$.  
\end{proposition}

\begin{proof}
We check  the conditions of Theorem \ref{thm:left_lifting}. Condition (i) is given by assumption and condition (ii) holds by the facts related to $\C(\B)_{\mathrm{inj}}$ that we recalled above. We check condition (iii): If $G$ is a projective generator of $\A$, then the set of all complexes of the form $D^{n}(G):=(0\rightarrow G=G\rightarrow 0)$ that are concentrated in degrees $n$ and $n-1$, is a set of projective generators for $\Ch(\A)$. Since each complex in $D^{n}(G)$ is acyclic and $i$ is exact, we have that $i$ maps the projective generator $\oplus_{n}D^{n}(G)$ of $\Ch(\A)$ to an acyclic complex in $\Ch(\B)$. 
Also, condition (iv) asks for the injective objects in $\Ch(\A)$, which are split exact complexes of injectives, to be mapped to acyclic complexes in $\Ch(\B)$, which holds since $i$ is a functor. Finally, the left-transferred model on $\A$ along $i$ coincides with $\Ch(\A)_{\mathrm{inj}}$ since $\Ch_{\mathrm{ac}}(\A)= i^{-1}\Ch_{\mathrm{ac}}(\B)$ (same argument as in the proof of Proposition  \ref{prop:transfer_trivial_proj}). 
\end{proof}

\section{Lifting recollements of abelian categories}
\label{sec:Lifting_recoll}

In this section we show that there is a conceptual homotopical framework on a recollement of abelian categories which naturally induces a recollement between the corresponding homotopy categories. We restrict to recollements of Grothendieck categories, i.e.\ the abelian categories of the recollement are Grothendieck. We refer to \cite{ParraVitoria} for studying when the abelian categories in a recollement are Grothendieck.

We first define the classes of abelian model structures that we will work with.

\begin{definition}
\label{def:proj_inj_models}
An abelian model structure $\M$ with Hovey triple $(\C_{\M},\W_{\M},\F_{\M})$ is called \textbf{projective} if all objects in $\M$ are fibrant. Dually, it is called \textbf{injective} if all objects in $\M$ are cofibrant.
\end{definition}

Since recollements of abelian and of triangulated categories involve fully faithful functors, we will also need the following notion.

\begin{definition}
\label{def:right_derived_embedding}
Let $l\colon\A\leftrightarrows\B\colon r$ be a Quillen adjunction between model categories. If the right adjoint $r$ is fully faithful and its total right derived functor $\textbf{R}r$ is fully faithful too, we call the functor $r\colon \B\rightarrow \A$ a \textbf{right derived embedding}. Dually, in case the left adjoint $l$ is fully faithful and its total left derived functor $\mathbf{L}l$ is fully faithful too, we call $l$ a \textbf{left derived embedding}.
\end{definition}

For the Quillen adjoint triples that we work with in the rest of the paper we have the following more compact definition.

\begin{definition}
\label{defn:derembedding}
Let $\mathrm{Q_{tr}}(\A_{1/2},\B_{1/2},\alpha)$ be a Quillen adjoint triple, where $f,f',g,g'$ are identity maps, $i=i'$, $\alpha=\mathrm{id}$ and $\A_1$ (resp., $\B_1$) shares the same weak equivalences with $\A_2$ (resp., $\B_2$). The functor $i\colon \A\to \B$ is called a {\bf derived embedding} if the functor $i\colon \A_1\to \B_1$ is a right-derived embedding, or equivalently, the functor $i\colon \A_2\to \B_2$ is a left derived embedding (cf. diagram (\ref{eq:proinj_triple}) and Proposition~\ref{prop:left_right_derived_embedding}).   
\end{definition}

The next Proposition describes (left/right) derived embeddings in a typical case involving (left/right) transferred model structures, and will be useful in the sequel.

\begin{proposition}
\label{prop:lifting_of_induced}
Consider an adjoint triple $(q,i,p)$ between abelian categories,
\begin{equation}
\xymatrix@C=3pc{
\A\ar[r]^-{i} & \B\ar@/_1.2pc/[l]_-{q} \ar@/^1.2pc/[l]^-{p} 
} 
\nonumber
\end{equation}
where $i$ is fully faithful.
Assume that $\B$ admits two hereditary abelian model structures $\B_{1}$ and $\B_{2}$, that share the same weak equivalences and are Quillen equivalent via the identity functor.
If the right-lifted and the left-lifted abelian model structures along the functor $i$ exist on $\A$, and are denoted by $\A_{1}$ and $A_{2}$ respectively, then the following hold:
\begin{itemize}
\item[(i)] There  exists a Quillen adjoint triple $\mathrm{Q_{tr}}(\A_{1/2},\B_{1/2},\mathrm{id})$.
\item[(ii)] $i\colon \A_1\to \B_1$ is a right-derived embedding if and only if for any object $X$ in $\A_{1}$ the unit map $\eta_{QiX}\colon Q_{\B_{1}}iX\rightarrow (i\circ q)(Q_{\B_{1}}iX)$ is a weak equivalence in $\B_{1}$.

\item[(iii)] $i\colon \A_2\to \B_2$ is a left-derived embedding if and only if for any object $X$ in $\A_{2}$ the counit map $\epsilon_{RiX}\colon  (i\circ p)(R_{\B_{2}}iX)\rightarrow R_{\B_{2}}iX$ is a weak equivalence in $\B_{2}$.
\end{itemize}
In case (ii) or (iii) holds, $i$ is a derived embedding for the Quillen adjoint triple $\mathrm{Q_{tr}}(\A_{1/2},\B_{1/2},\mathrm{id})$.
\end{proposition}
\begin{proof}
(i) We observe that the classes of weak equivalences (and of trivial objects) for both $\A_{\pi}$ and $\A_{\iota}$ coincide as preimages under $i$ of those in $\B_{1}$ and $\B_{2}$, thus it is easy to see that the identity map from $A_{\pi}$ to $\A_{\iota}$ is a Quillen equivalence. 
Hence we have a Quillen adjoint triple $\mathrm{Q_{tr}}(\A_{1/2},\B_{1/2},\alpha)$, as in Definition \ref{def:Quillen_adjoint_triple} where $f, f', g, g'$ are identity morphisms, $i=i'$ and $\alpha=\mathrm{id}$.
(ii) For an object $X$ in $\A$ consider the following commutative diagram in $\B$:
\begin{equation}
\label{eq:unit_map_right}
  \xymatrix@C=3pc{
(i\circ q)(Q_{\B_{1}}iX)  \ar[r] ^-{iq(\mathsf{q}_{iX}^{\B_{1}})}    &  iqi(X) \\
Q_{\B_{1}}iX \ar[u]_-{\eta_{QiX}}   \ar[r]_-{\mathsf{q}_{iX}^{\B_{1}}}^-{\sim} & iX  \ar[u] ^-{\cong}
  }\!
\nonumber
\end{equation}
and recall from \ref{prop:fully_faithful_right_derived} that $i$ is a right derived embedding if and only if $q$ maps the weak equivalence $\mathsf{q}_{iX}^{\B_{1}}\colon Q_{\B_{1}}iX\rightarrow iX$ to a weak equivalence $q(\mathsf{q}_{iX}^{\B_1})$ in $\A_{1}$. By the definition of weak equivalences in a right-transferred model structure, we obtain that $q(\mathsf{q}_{iX}^{\M_1})$ is a weak equivalence if and only if the top map in the above commutative diagram is. Now by the two-out-of-three property for weak equivalences the result follows. The proof of (iii) is dual to that of (ii). The last assertion follows by Definition \ref{defn:derembedding}.
\end{proof}

The following setup sets the stage for our general lifting theorem, proved in \ref{thm:main_lift}.

\begin{setup}
\label{Setup}
Consider a recollement of Grothendieck categories 
\begin{equation*}
  \xymatrix@C=4pc{
\A   \ar@<0.0ex>[r] |-{\, i\,}   &  
  \B \ar@/^1.4pc/[l] ^-{p}    \ar@/_1.4pc/[l] _-{q}  \ar[r] |-{\, e \,}&
  \C \ar@/_1.4pc/[l]  _-{l}  \ar@/^1.4pc/[l] ^-{r}
  }\!
\tag*{$\mathsf{R_{ab}}(\A,\B,\C)$}
\end{equation*}
such that the following hold:
\begin{itemize}
\item $\B$ admits two hereditary abelian model structures with Hovey triples $\B_{\mathsf{proj}}:=(\C_{\B},\W_{\B},\B)$ and $\B_{\mathsf{inj}}:=(\B,\W_{\B},\F_{\B})$ which are Quillen equivalent via the identity functor.
\item  $\C$ admits two hereditary abelian model structures with Hovey triples $\C_{\mathsf{proj}}:=(\C_{\C},\W_{\C},\C)$ and $\C_{\mathsf{inj}}:=(\C,\W_{\C},\F_{\C})$ which are Quillen equivalent via the identity functor.
\item The functor $e\colon \B\rightarrow\C$ preserves trivial objects, i.e., $e(\W_{\B})\subseteq\W_{\C}$.
\item The right transfer of $\B_{\mathsf{proj}}$ along the functor $i\colon \A\rightarrow\B$ exists. We denote it by $\A_{\pi}$.
\item The left transfer of $\B_{\mathsf{inj}}$ along the functor $i\colon \A\rightarrow\B$ exists. We denote it by $\A_{\iota}$.
\end{itemize}
\end{setup}

\begin{remark} 
\label{rem:comments_on_setup}
The following comments concerning Setup \ref{Setup} are in order.
\begin{itemize}
\item[(i)] By definition of Hovey triples (\ref{def:Hovey_triples}), the existence of $\B_{\mathsf{proj}}$ and $\C_{\mathsf{proj}}$ implies, in particular, the existence of complete cotorsion pairs $(\Proj(\B),\B)$ and $(\Proj(\C),\C)$, in other words, the Grothendieck categories $\B$ and $\C$ have enough projective objects.

\item[(ii)] Notice that the identity functor on $\B$ is a left Quillen functor $\mathrm{id}\colon\B_{\mathsf{proj}}\rightarrow \B_{\mathsf{inj}}$ and similarly for $\C$.

\item[(iii)] The class $\W_{\B}$ is commonly the class of trivial objects for the abelian model structures $\B_{\mathsf{proj}}$ and $\B_{\mathsf{inj}}$, thus they share the same weak equivalences (by the characterization of weak equivalences given in \cite[Prop.~2.3]{Gilsurvey}) and the same homotopy category, which we denote by $\Ho(\B)_{\mathsf{proj/inj}}$. The same comment applies to $\C_{\mathsf{proj}}$ and $\C_{\mathsf{inj}}$ which share the homotopy category $\Ho(\C)_{\mathsf{proj/inj}}$.

\item[(iv)] The right, resp., left,  transfer of $\B_{\mathsf{proj}}$, resp., $\B_{\mathsf{inj}}$ along $i$, assuming they exist, they share the same class of trivial objects, which is $i^{-1}(\W_{\B})$, thus they share the same homotopy category, which we denote by $\Ho(\A)_{\pi/\iota}$.

\item[(v)] Setup \ref{Setup} requires the existence of certain left and right transferred model structures. 
Theorems \ref{thm:right_lifting}/\ref{thm:left_lifting} give sufficient conditions for their existence in general, and usually it is hard to check if the acyclicity conditions hold. Working with projective and injective abelian model structures, as in Setup \ref{setup_lifting}, makes this an easier task. Indeed, in this case the right acyclicity condition from \ref{thm:right_lifting} reads $i(\Proj(\A))\subseteq \W_{\B}$, and dually, the left acyclicity condition from \ref{thm:left_lifting} reads $i(\Inj(\A))\subseteq \W_{\B}$. In applications these conditions are usually easy to check (so the above setup is not very restricting).
\end{itemize}
\end{remark}


We are now ready to state and prove our main result on lifting of recollements of abelian categories. 

\begin{theorem}
\label{thm:main_lift}
Let $\mathsf{R_{\mathsf{ab}}}(\A,\B,\C)$ be a recollement of Grothendieck categories satisfying Setup \ref{Setup}. Then there exists a diagram of triangulated categories 
\begin{equation}
\label{eq:dia_in_main}
\xymatrix@C=4pc{
\Ho(\A)_{\pi/\iota}  \ar@<0.0ex>[r] |-{\, \textbf{R}i\cong\textbf{L}i\,}   &  
\Ho(\B)_{\mathsf{\tiny{proj/inj}}}   \ar@/^1.4pc/[l] ^-{\textbf{R}p}    \ar@/_1.4pc/[l] _-{\textbf{L}q}  \ar[r] |-{\, \textbf{R}e\cong\textbf{L}e \,}&
\Ho(\C)_{\mathsf{proj/inj}} \ar@/_1.4pc/[l]  _-{\textbf{L}l}  \ar@/^1.4pc/[l] ^-{\textbf{R}r}
  }
\end{equation}
where $\Ho(\A)_{\pi/\iota}$ is the common homotopy category of the right and left transferred abelian model structures along the functor $i$. This diagram is a recollement if and only if the following conditions hold:
\begin{enumerate}
\item[(i)] The functor $i\colon \A\to \B$ is a derived embedding (as in Defintion \ref{defn:derembedding}).

\item[(ii)]  For all cofibrant objects $X$ in $\B_{\mathsf{proj}}$ that belong to $\Ker\textbf{R}e$, the unit map $X\rightarrow (i\circ q)X$ is a weak equivalence in $\B_{\mathsf{proj}}$. 
\end{enumerate}
Under the presence of (i), condition (ii) is equivalent to the following: $(ii)'$ For all fibrant objects $Y$ in $\B_{\mathsf{inj}}$ that belong to $\Ker\textbf{L}e$, the counit map $(i\circ p)Y\rightarrow Y$ is a weak equivalence in $\B_{\mathsf{inj}}.$ 
\end{theorem}

\begin{proof}
To obtain the right-hand side of (\ref{eq:dia_in_main}), we consider the quotient functor $e\colon \B\to \C$, which induces a commutative diagram (expressing the identity natural transformation of the functor $e$)
\begin{equation*}
  \xymatrix@C=3pc{
\B_{\mathsf{proj}} \ar[r]^-{e}  \ar[d]^-{\mathsf{id}} &  \C_{\mathsf{proj}} \ar[d]^-{\mathsf{id}} \\
B_{\mathsf{inj}} \ar[r]^-{e} &  \C_{\mathsf{inj}}
  }\!
\end{equation*}
where the vertical identity functors are left Quillen. Since $e$ is assumed to preserve trivial objects, we also deduce that the functor on top of the above diagram is right Quillen and that the one on the bottom is left Quillen. Therefore we obtain a Quillen adjoint triple $\mathrm{Q_{tr}}(\B_{\mathsf{proj/inj}},\C_{\mathsf{proj/inj}},\mathrm{id})$ (as in Definition \ref{def:Quillen_adjoint_triple}), which by Proposition~\ref{prop:special_triple} induces an adjoint triple at the level of homotopy categories, hence we obtain the right-hand side of (\ref{eq:dia_in_main}). In addition, from Proposition~\ref{prop:left_right_fully_faithful} we deduce that the functors $\mathbf{L}e$ and $\mathbf{R}e$ in this adjoint triple are fully faithful.

For the left hand-side of (\ref{eq:dia_in_main}) we work as follows. The existence of the right-lifted and left-lifted model structures on $\A$ along the functor $i\colon \A\rightarrow \B$ as stated in Setup \ref{Setup}, induces, 
by Proposition \ref{prop:lifting_of_induced}(i), a Quillen adjoint triple $\mathrm{Q_{tr}}(\A_{\pi/\iota},\B_{\mathsf{proj/inj}},\mathrm{id})$, which
by Proposition~\ref{prop:special_triple} induces an adjoint triple at the level of homotopy categories; thus we obtain the left-hand side of diagram (\ref{eq:dia_in_main}).

In an expanded form, the diagram we have obtained at the level of homotopy categories is the following
\begin{equation}
\label{eq:expanded_dia}
  \xymatrix@C=3pc{
\Ho(\A)_{\pi} \ar[d]^-{\mathbf{L}\mathrm{id}}_{\cong} \ar[r]^-{\mathbf{R}i} & \Ho(\B)_{\mathsf{proj}}  \ar[r]^-{\mathbf{R}e}  \ar[d]^-{\mathbf{L}\mathrm{id}}_{\cong} &  \Ho(\C)_{\mathsf{proj}} \ar[d]^-{\mathbf{L}\mathrm{id}}_{\cong} \\
\Ho(\A)_{\iota} \ar[r]^-{\mathbf{L}i}& \Ho(\B)_{\mathsf{inj}}  \ar[r]^-{\mathbf{L}e} & \Ho(\C)_{\mathsf{inj}}
  }\!
\end{equation}

From this we deduce the following diagram of triangulated categories 
    \begin{equation}
      \xymatrix@!=0.8pc{   
      \Ho(\A)_{\pi} \, \ar[rrr]^{\mathbf{R}i}  \ar[dd]_{\mathbf{L}\mathrm{id}} \ar[drr]^-{j'}  & & &  \Ho(\B)_{\mathsf{proj}}  \ar[dd]^{\mathbf{L}\mathrm{id}}        \\
          {} & & \ker\mathbf{R}e \ \ \ar@{>->}[ur]^{j} \ar[dd]     &           \\
          \Ho(\A)_{\iota} \ar[drr]^-{k'}  \, 
          \ar@{..>}[rrr]^(0.40){\mathbf{L}i} & &  &  \Ho(\B)_{\mathsf{inj}}  
          \\
     &     &  \ker\mathbf{L}e \ \  \ar@{>->}[ur]^{k} & 
       }
       \nonumber
\end{equation}
where $j$ and $k$ denote the canonical inclusions, while $j'$ and $k'$ are induced by universal properties of kernels. Also, the canonical vertical functor between the kernels is an equivalence since the two right-most vertical functors in diagram (\ref{eq:expanded_dia}) are equivalences. We make the following observations: (a) $j'$ is an equivalence if and only if $k'$ is an equivalence, in which case diagram (\ref{eq:dia_in_main}) is a recollement of triangulated categories (and conversely, if diagram (\ref{eq:dia_in_main}) is a recollement then $j'$ and $k'$ are equivalences); (b) $\mathbf{R}i$ is fully faithful if and only if $j'$ is fully faithful, and similarly, $\mathbf{L}i$ is fully fathful if and only if $k'$ is fully faithful; (c) If $j'$ is fully faithful then its essential image consists of all objects $X$ in $\ker\mathbf{R}e$ such that the (derived) adjunction unit $X\to (\mathbf{R}i\circ\mathbf{L}q)(X)$ is an isomorphism in $\Ho(\B)_{\mathsf{proj}}$; (d) By a standard reduction to cofibrant objects \cite[Proposition~1.3.13]{hoveybook}, condition (ii) can be equivalently stated \textit{for all} objects $X\in\ker\mathbf{R}e$ and it is also equivalent to asking that for all objects $X\in\ker\mathbf{R}e$ the (derived) adjunction unit $X\to (\mathbf{R}i\circ\mathbf{L}q)(X)$ is an isomorphism.

In total, if (\ref{eq:dia_in_main}) is a recollement, by combining (a), (b), (c) and (d) we obtain that conditions (i) and (ii) are hold. Conversely, if (i) and (ii) hold, then by (b) $j'$ and $k'$ are fully faithful and by (c) and (d) they are essentially surjective too;  thus by (a) we deduce that (\ref{eq:dia_in_main}) is a recollement.

The fact that the the existence of the recollement in (\ref{eq:dia_in_main}) is equivalent to the validity of  (i)+(ii)' is left to the reader. 
\end{proof}

\section{Applications}
\label{sec:applications}

\subsection{Stable categories of Gorenstein projectives}

\subsubsection{Triangular matrix algebras}
\label{subsub:triangular}

In this section we investigate liftings of recollements of module categories to associated homotopy categories of Gorenstein projective and Gorenstein injective modules.

Let $R$ be a Iwanaga-Gorenstein ring, that is, $R$ is two-sided noetherian and has finite injective dimension on both sides. From Hovey \cite[Theorem~8.6]{hovey2002} (see also Holm \cite{holm}) there exist the following two hereditary Hovey triples on the category of left $R$--modules, $\Mod{R}$.

\begin{itemize}
\item $(\Mod{R})_{\mathsf{GP}}:=(\GProj{(R)},\mathcal{P}^{<\infty}(R),\Mod{R})$, where $\GProj(R)$ denotes the class of Gorenstein projective $R$--modules and $\mathcal{P}^{<\infty}(R)$ denotes the class of $R$--modules of finite projective dimension. 

\item $(\Mod{R})_{\mathsf{GI}}:=(\Mod{R},\mathcal{I}^{<\infty}(R),\GInj(R))$, where $\GInj(R)$ denotes the class of Gorenstein injective $R$--modules and $\mathcal{I}^{<\infty}(R)$ denotes the class of $R$--modules of finite injective dimension.
\end{itemize}
According to Definition \ref{def:proj_inj_models},  $(\Mod{R})_{\mathsf{GP}}$ is a projective model structure while its dual counterpart $(\Mod{R})_{\mathsf{GI}}$ is an injective model structure. 
The cotorsion pairs that constitute all the Hovey triples involved are known to be each generated by a set, see  \cite[Thm.~8.3/8.4]{hovey2002}, hence the associated model structures are cofibrantly generated by \cite[Lemma~6.7]{hovey2002}. Over an Iwanaga-Gorenstein ring, the classes of trivial objects for the above two model  structures coincide, that is, $\mathcal{P}^{<\infty}(R)=\mathcal{I}^{<\infty}(R)$, hence the classes of weak equivalences also coincide; we just denote them by $\W$. For the (common) associated homotopy category we use the notation 
\[
(\Mod{R})[\W^{-1}]:=\Ho((\Mod{R})_{\mathsf{GP} / \mathsf{GI}}):=\Ho(R)_{\mathsf{GP} / \mathsf{GI}}.
\]

It is easy to see that the identity functor on $\Mod(R)$ is a Quillen equivalence between $(\Mod{R})_{\mathsf{GP}}$ and $(\Mod{R})_{\mathsf{GI}}$ (here the identity is left Quillen from the projective model to the injective model).
In fact, the homotopy category of $(\Mod{R})_{\mathsf{GP}}$ is equivalent, as a triangulated category, to the stable category of Gorenstein projective $R$--modules, and dually for $(\Mod{R})_{\mathsf{GI}}$. The reader may consult \cite[Thm.~3.7/3.9]{Dal_Estrada_Holm} for detailed arguments and a generalization.

In the next example we consider transfers of these models along certain ring homomorphisms.

\begin{example}
\label{ex:Gor_morphism_triple}
Let $\phi\colon R\rightarrow S$ be a ring epimorphism where $R$ is Gorenstein and $S$ has finite projective dimension as a left $R$--module.
We consider the adjoint triple
\begin{equation}
\label{eq:Gor_example_triple}
\xymatrix@C=4pc{
\Mod{S}\ar[r]^{\, i:=\phi_{*}\, } & \Mod{R}\ar@/_1.4pc/[l]_-{q:=S\otimes_{R}-} \ar@/^1.4pc/[l]^-{p:=\Hom_{R}(S,-)} 
} 
\nonumber
\end{equation}
and we equip $\Mod{R}$ with two (Quillen equivalent) hereditary abelian model structures, $(\Mod{R})_{\mathsf{GP}}$ and $(\Mod{R})_{\mathsf{GI}}$, as explained above.
We recall the fact that $\phi$ is a ring epimorphism if and only if the restriction of scalars functor $\phi_{*}\colon\Mod{S}\rightarrow\Mod{R}$ is fully faithful; see \cite[Ch.XI]{stenstrom}.
In order to apply Theorem~\ref{thm:right_lifting} and Theorem~\ref{thm:left_lifting} (see also Remark~\ref{rem:comments_on_setup} (v)), we check that the following hold:
\begin{itemize}
\item[(i)] The functor $q$ maps cofibrations to monomorphisms. Indeed, let $0\rightarrow A\rightarrow B\rightarrow G\rightarrow 0$ be a short exact sequence with $G\in\GProj{R}$, i.e., a  cofibration in $(\Mod{R})_{\GProj}$, then the desired condition holds if $\Tor_{1}^{R}(S,G)=0$. Note that since $R$ is Iwanaga-Gorenstein, $G$ is Gorenstein flat (by \cite[Proposition~3.4]{holm}), and that $S\in\class P^{<\infty}(R)=\I^{<\infty}(R)$, thus from \cite[Theorem~3.14]{holm} we obtain the desired result.

\item[(ii)] $\phi_{*}(\Proj{S})\subseteq \class P^{<\infty}(R)$. This is clear since $S\in\class P^{<\infty}(R)$.

\item[(iii)] The functor $i$ maps the generator $S$ to a module of finite projective (=finite injective) dimension, by the assumption on $\phi\colon R\rightarrow S$.

\item[(iv)] $\phi_{*}(\Inj{S})\subseteq \I^{<\infty}(R)$. To see this consider an injective $S$--module $I$ and the natural isomorphism of functors $\RHom_{R}(-,\phi_{*}(I))\cong\RHom_{S}(S\otimes_{R}^{\textbf{L}}-,I)$ (for this formula see for instance \cite[Corollary~5.16]{CFH}). Since $I$ is injective and $S$ has finite flat dimension over $R$, the assertion follows.
\end{itemize}

This implies that the right-lifted and the left-lifted hereditary abelian model structures along the functor $i$ exist. We denote them by $(\Mod{S})_{\pi}$ and $(\Mod{S})_{\iota}$ respectively. The homotopy category of both model structures is the localization of $\Mod{S}$ to the class $i^{-1}(\W)$ (where $\W$ denotes the class of weak equivalences over the Iwanaga-Gorenstein ring $R$) and we denote it by $\Ho(S)_{\pi/\iota}$. Thus, we have the following commutative diagram (expressing the identity natural transformation of the functor $i$)
\begin{equation*}
  \xymatrix@C=3pc{
(\Mod{S})_{\pi} \ar[r]^-{i}  \ar[d]^-{\mathsf{id}} &  (\Mod{R})_{\GProj} \ar[d]^-{\mathsf{id}} \\
(\Mod{S})_{\iota} \ar[r]^-{i} &  (\Mod{R})_{\GInj}
  }\!
\end{equation*}
where the functor on top is right Quillen by construction of the right-lifted model structure, and the bottom functor is left Quillen by construction of the left-lifted model structure. Also, the vertical identity functors are Quillen equivalences. Hence we obtain a Quillen adjoint triple $\mathrm{Q_{tr}}((\Mod{S})_{\pi/\iota},(\Mod{R})_{\GProj/ \GInj},\mathrm{id})$ which by Proposition \ref{prop:special_triple} induces an adjoint triple of triangulated categories. 

\begin{equation}
\label{eq:adjoint_triple_g_example}
  \xymatrix@C=4pc{
\Ho(S)_{\pi/\iota}   \ar@<0.0ex>[r] ^-{\textbf{L}i\, \cong\,  \textbf{R}i}   &
\Ho(R)_{\mathsf{GP} / \mathsf{GI}}. \ar@/^1.5pc/[l] ^-{\textbf{R}p}    \ar@/_1.5pc/[l] _-{\textbf{L}q} 
  }\!  
\end{equation}
We point out that similar adjunctions have been studied in \cite{Bergh, OPS}. 
In our case, if the ring $S$ is Iwanaga-Gorenstein we obtain $\Ho(S)_{\pi/\iota}=\Ho(S)_{\mathsf{GP} / \mathsf{GI}}$ (since in this case all the involved models share the same class of trivial objects). But in general the cofibrant objects in $(\Mod{S})_{\pi}$ are given by the class $\leftperp{i^{-1}(\mathcal{P}^{<\infty}(R))}$ which we call relative Gorenstein projective $S$--modules (relative to the right adjoint $\phi_{*}:=i\colon \Mod{S}\rightarrow \Mod{R}$). In case $S$ is Gorenstein or $\phi=\mathrm{id}$ this class is just $\mathsf{GProj}(R)$. Similar comments apply to (relative) Gorenstein injectives.
\end{example}

\begin{example} (\textnormal{Gorenstein triangular matrix rings})
\label{ex:Gor_triang_matrix_RHS}
Consider the ring 
\[
R:=\begin{pmatrix} 
A & {_AM_B} \\
0 & B
\end{pmatrix}
\]
where $A$ and $B$ are Iwanaga-Gorenstein rings and $M$ is a $A$-$B$-bimodule which is finitely generated on both sides and has finite projective dimension on both sides. Then by \cite[Theorem~2.2, Lemma~2.3]{XiongZhang} the ring $R$ is Iwanaga-Gorenstein. For $e=\bigl(\begin{smallmatrix}
1 & 0 \\
0 & 0
\end{smallmatrix}\bigr)$, 
the recollement associated to the idempotent $e$ (see \cite[Example~3.10]{Psa_survey}) takes the form
\[
  \xymatrix@C=4pc{
\Mod{B}   \ar@<0.0ex>[r] |-{\, i\,}   &  
  \Mod{R} \ar@/^1.4pc/[l] ^-{p}    \ar@/_1.4pc/[l] _-{q}  \ar[r] |-{\, e \,}&
  \Mod{A}. \ar@/_1.4pc/[l]  _-{l}  \ar@/^1.4pc/[l] ^-{r}
  }
\]
where the functors $l, e, i, q$ are exact. The quotient functor $e\colon \Mod{R}\to \Mod{A}$ induces a commutative diagram (the identity natural transformation of $e$)
\begin{equation*}
  \xymatrix@C=3pc{
(\Mod{R})_{\mathsf{GP}} \ar[r]^-{e}  \ar[d]^-{\mathsf{id}} &  (\Mod{A})_{\mathsf{GP}} \ar[d]^-{\mathsf{id}} \\
(\Mod{R})_{\mathsf{GI}} \ar[r]^-{e} &  (\Mod{A})_{\mathsf{GI}}
  }\!
\end{equation*}
where the vertical identity functors are Quillen equivalences. We also claim that the top functor is right Quillen and that the bottom functor is left Quillen. We need to check that the functor $e$ on top sends trivial fibrations to trivial fibrations, equivalently, if $X$ is an $R$-module with $\pd_RX<\infty$ then we need to show that $\pd_Ae(X)<\infty$. Using that $(l,e)$ in an adjoint pair and $l$ is exact, it follows that $e$ preserves modules of finite injective dimension. By the Gorensteiness of $R$ and $A$, we obtain that $e$ preserves modules of finite projective dimension as well, thus the claim follows. 
By a similar argument it follows that the functor $e$ on the bottom is left Quillen. 
Therefore  by Proposition~\ref{prop:special_triple} we obtain an adjoint triple of triangulated categories

\begin{equation}
\label{eq:adjoint_triple_g_example_RHS}
  \xymatrix@C=4pc{
\Ho(R)_{\mathsf{GP} / \mathsf{GI}}   \ar@<0.0ex>[r] ^-{\textbf{L}e\, \cong\,  \textbf{R}e}   &
\Ho(A)_{\mathsf{GP} / \mathsf{GI}}. \ar@/^1.5pc/[l] ^-{\textbf{R}r}    \ar@/_1.5pc/[l] _-{\textbf{L}l} 
  }\! 
\end{equation}

We further claim that the functor $\mathbf{L}l$ is fully faithful (which is equivalent by Proposition~\ref{prop:left_right_derived_embedding} to $\mathbf{R}r$ being fully faithful). We may use Proposition \ref{prop:fully_faithful_left_derived}(iii) to show that $\mathbf{L}l$ is fully faithful. Take $X$ a cofibrant object in $\Mod{A}$, i.e.,\ $X$ lies in $\GProj{A}$. Then a fibrant replacement of $l(X)$ means that we only consider the identity map $\mathsf{Id}_{l(X)}\colon l(X)\to l(X)$ since in the Gorenstein projective abelian model structure all objects are fibrant. Thus, applying the functor $e$ we clearly get a weak equivalence since  $e\circ l\simeq \mathsf{Id}_{\Mod{A}}$. We conclude that $\mathbf{L}l$ is fully faithful. Hence diagram (\ref{eq:adjoint_triple_g_example_RHS}) constitutes the right hand side of a recollement of triangulated categories.
\end{example}

Now, from Example~\ref{ex:Gor_morphism_triple}, we deduce the existence of two Quillen equivalent model structures on $\Mod{B}$, that give rise to a Quillen adjoint triple as in (\ref{eq:adjoint_triple_g_example}). In the next result we show that putting together both the adjoint triples obtained induces a recollement of triangulated categories (for an analogous result cf.~\cite[Theorem~3.5]{Zhang}).

\begin{theorem}\label{thm:triangular_lifting}
Consider the ring $R:=\bigl(\begin{smallmatrix}
A & {_AM_B} \\
0 & B
\end{smallmatrix}\bigr)$ where $A$ and $B$ are Iwanaga-Gorenstein rings and $M$ is a $A$-$B$-bimodule which is finitely generated on both sides and has finite projective dimension on both sides. Then the recollement associated to the idempotent $e=\bigl(\begin{smallmatrix}
1 & 0 \\
0 & 0
\end{smallmatrix}\bigr)$,
lifts to a recollement of triangulated categories 
\begin{equation}
\label{eq:Gor_ex_big_diagram}
\xymatrix@C=4pc{
\Ho(B)_{\pi/\iota}  \ar@<0.0ex>[r] |-{\, \textbf{R}i\cong\textbf{L}i\,}   &  
\Ho(R)_{\mathsf{\tiny{GP/GI}}}   \ar@/^1.4pc/[l] ^-{\textbf{R}p}    \ar@/_1.4pc/[l] _-{\textbf{L}q}  \ar[r] |-{\, \textbf{R}e\cong\textbf{L}e \,}&
\Ho(A)_{\mathsf{GP/GI}} \ar@/_1.4pc/[l]  _-{\textbf{L}l}  \ar@/^1.4pc/[l] ^-{\textbf{R}r}
  }
\end{equation}
where $\Ho(B)_{\pi/\iota}$ is the \textnormal{(}common\textnormal{)} homotopy category of the right and left transferred abelian model structures along the forgetful functor $i$. In this case, the homotopy category $\Ho(B)_{\pi/\iota}$ coincides with $\Ho(B)_{\mathsf{\tiny{GP/GI}}}$.
\end{theorem}

\begin{proof}
To obtain a diagram of triangulated categories as in (\ref{eq:Gor_ex_big_diagram}) we put together diagram (\ref{eq:adjoint_triple_g_example_RHS}) from Example~\ref{ex:Gor_triang_matrix_RHS} (which is the right hand side of a recollement) and diagram (\ref{eq:adjoint_triple_g_example}) from Example~\ref{ex:Gor_morphism_triple}. It remains to prove that the two conditions of Theorem~\ref{thm:main_lift} are satisfied.

To prove the validity of condition (i), it suffices to show that the functor\\ $i\colon(\Mod{B})_{\pi}\rightarrow (\Mod{R})_{\mathsf{GP}}$ is a right derived embedding, or equivalently, from Proposition~\ref{prop:lifting_of_induced}, to show that for any object $X$ in $\Mod{B}$, the unit map $$\eta_{QiX}\colon (i\circ q)(QiX)\rightarrow QiX$$ is a weak equivalence in $(\Mod{R})_{\mathsf{GP}}$, where $0\rightarrow W\rightarrow QiX\xrightarrow{\sim} iX\rightarrow 0$ is a cofibrant replacement (Gorenstein projective approximation) of $iX$ in $(\Mod{R})_{\mathsf{GP}}$. We observe that $eW\cong eQiX$, thus $eQiX$ is a trivial object in $(\Mod{A})_{\mathsf{GP}}$. Therefore, $(l\circ e) QiX$ is a trivial object $(\Mod{R})_{\mathsf{GP}}$ (since $l$ preserves trivial objects). The assertion now follows from the canonical short exact sequence 
\[
\xymatrix@C=2pc{
0 \ar[r] &  (l\circ e)(QiX) \ar[r]^-{\epsilon_{QiX}} & QiX \ar[r]^-{\eta_{QiX}} & (i\circ q)(QiX) \ar[r] & 0. 
} \]

To prove the validity of condition (ii) from Theorem \ref{thm:main_lift}, let $M$ be cofibrant (Gorenstein projetive) $R$--module in $\Ker\mathbf{R}e$, and consider the short exact sequence 
\[
\xymatrix@C=2pc{
0 \ar[r] &  (l\circ e)(M) \ar[r]^-{\epsilon_{M}} & M \ar[r]^-{\eta_{M}} & (i\circ q)(M) \ar[r] & 0. 
} \]
We want to prove that ${\eta_{M}}$ is a weak equivalence in  $(\Mod{R})_{\mathsf{GP}}$. For this it suffices to show that  $(l\circ e)(M)$  is a trivial object in $(\Mod{R})_{\mathsf{GP}}$. To that end, note that $eX=\mathbf{R}e(X)$ is a trivial object by assumption and that the functor $l$ is exact and preserves projectives, thus it maps $\mathcal{P}^{<\infty}(A)$ to $\mathcal{P}^{<\infty}(R)$, i.e. it preserves trivial objects; this shows that $(l\circ e)(M)$ is a trivial object.

Finally, we prove that $\Ho(B)_{\pi/\iota}$ coincides with $\Ho(B)_{\mathsf{\tiny{GP/GI}}}$. It is enough to show that the trivial classes of objects for these two abelian model structures coincide, i.e., that $i^{-1}(\mathcal{P}^{<\infty}(R))=\mathcal{P}^{<\infty}(B)$. Let $Y$ be a $B$-module such that $\pd_Ri(Y)<\infty$. Using that $(q, i)$ is an adjoint pair with both functors exact and that $q\circ i\simeq  \mathsf{Id}_{\Mod{B}}$, we deduce that $\pd_BY<\infty$. Suppose now that the $B$-module $Y$ has finite projective dimension. Since $B$ is Iwanaga-Gorenstein, $Y$ also has finite injective dimension. Again by the adjunction $(q,i)$ we obtain that the injective dimension of the $R$--module $i(Y)$ is finite, and since $R$ is Iwanaga-Gorenstein, we conclude that $\pd_Ri(Y)<\infty$, i.e., $Y\in\W_{B}$.
\end{proof}

\subsubsection{n-morphism categories}

Let $Q$ be a quiver with a finite number of vertices $Q_0$ and a finite number of arrows $Q_1$, let $\C$ be an abelian category with enough projective and injective objects and denote by $\Rep_{Q}(\C)$ the representations of $Q$ in $\C$. For any vertex $i$ in $Q$, there exist two canonical morphisms in $\C$, 
\[\bigoplus\limits_{\alpha:j\rightarrow i}X(j)\xrightarrow{\phi_{i}^{X}} X(i)\,\,\,\,\,\,\,\,\,\, \mbox{and}\,\,\,\,\,\,\,\,\,\, X(i)\xrightarrow{\psi_{i}^{X}} \prod\limits_{\alpha:i\rightarrow j}X(j).\]
Following \cite{HJ} for any class of objects $\X$ in $\C$ we consider two subcategories of $\Rep_{Q}(\C)$ as follows:
\begin{itemize}
\item[-] $\Phi(\X):=\{F\,|\, \forall c\in\C,\, \phi_{i}^{F}\,\, \mbox{is mono and}\, F(c), \Coker{\phi_{i}^{F}}\in \X, \, \forall i\in Q_0\}.$

\item[-] $\Psi(\X):=\{F\,|\, \forall c\in\C,\, \psi_{i}^{F}\,\, \mbox{is epi and}\, F(c), \Ker{\phi_{i}^{F}}\in \X, \, \forall i\in Q_0\}.$
\end{itemize}

Let $R$ be an Iwanaga-Gorenstein ring. The hereditary Hovey triples $(\Mod{R})_{\mathsf{GP}}$ and $(\Mod{R})_{\mathsf{GI}}$ that we recalled in \ref{subsub:triangular}, from \cite[Theorems~A and B]{HJ}, lift to the following hereditary Hovey triples on $\Rep_Q(\Mod{R})$,
\begin{itemize}
\item[-] $\Rep_{Q}(\Mod{R})_{\mathsf{GP}}:=(\Phi(\GProj{R}),\Rep_Q(\mathcal{P}^{<\infty}(R)),\Rep_{Q}(R))$,

\item[-]  $\Rep_{Q}(\Mod{R})_{\mathsf{GI}}:=(\Rep_{Q}(R),\Rep_{Q}(\I^{<\infty}(R)),\Psi(\GInj(R)))$,
\end{itemize}
which both have their associated complete cotorsion pairs generated by sets. In fact, from \cite[Theorem~3.5.1]{Esh} the class $\Phi(\GProj{R})$ is precisely the class of (categorical) Gorenstein projectives in $\Rep_{Q}(\Mod{R})$ and dually for $\Psi(\GInj(R))$; thus 
the homotopy category of $\Rep_{Q}(\Mod{R})_{\mathsf{GP}}$ (resp., $\Rep_{Q}(\Mod{R})_{\mathsf{GI}}$) is the stable category of Gorenstein projective (resp., injective) representations and there is also a Quillen equivalence between these by \cite[Thm.~4.8]{Dal_Estrada_Holm}.

Now, consider the Dynkin quiver $Q:=\mathbb{A}_n\colon$
\[
1\rightarrow 2 \rightarrow \cdots  \rightarrow n
\]
and the representation category $\Rep_{\mathbb{A}_n}(\Mod{R})$, also known as the {\bf $n$-morphism category} of the Iwanaga-Gorenstein ring $R$. In this case, the category $\Rep_{\mathbb{A}_n}(\Mod{R})$, also denoted by $\mathsf{Mor}_n(R)$, is equivalent to the module category $\Mod{\mathsf{T}_{n}(R)}$ where $\mathsf{T}_{n}(R)$ is the $n\times n$ triangular matrix ring:
\[
\mathsf{T}_{n}(R)=\begin{pmatrix}
 R & 0 & \cdots & 0\\
 R & R & \cdots  & 0\\
 \vdots & \vdots   & \ddots & \vdots  \\
  R & R & \cdots  & R\\
\end{pmatrix}
\]
Recall that the objects of the $n$-morphism category $\mathsf{Mor}_n(R)$ are sequences of the form 
\[
(X, f)\colon \, X_{1}\xrightarrow{f_1} X_{2}\xrightarrow{f_2}\cdots \to X_{n-1}\xrightarrow{f_{n-1}} X_{n}
\] 
such that all $X_{i}$ lie in $\Mod{R}$. Let $(X, f)$ and $(X', f')$ two objects in $\mathsf{Mor}_n(R)$. A morphism between them is a commutative diagram$\colon$
\[
\xymatrix{
 X_1 \ar[d]^{a_1} \ar[r]^{f_1 \ } & X_2 \ar[d]^{a_2} \ar[r]^{f_2 \ } & \cdots \ar[r]^{f_{n-1} \ } & X_{n} \ar[d]^{a_{n}} \\
 X_1' \ar[r]^{f_1' \ } & X_2' \ar[r]^{f_2' \ } & \cdots \ar[r]^{f_{n-1}' \ } & X_{n}' }
\]
that is,  $f_{i}'a_i=a_{i+1}f_{i}$ for all $1\leq i\leq n-1$, where $a_i\colon A_i\to A_i'$ are morphisms in $\Mod{R}$ for all $1\leq i\leq n$. 
From \cite[Example~4.9]{ladders} we have the following recollement:

\begin{equation}
\label{eq:heights}
\xymatrix@C=0.5cm{
\mathsf{Mor}_{n-1}(R) \ar[rrr]^{i} &&& \mathsf{Mor}_{n}(R)  \ar[rrr]^{e} \ar @/_1.5pc/[lll]_{q}  \ar @/^1.5pc/[lll]^{p} &&& \Mod{R} 
\ar @/_1.5pc/[lll]_{l} \ar
 @/^1.5pc/[lll]^{r}
 }
\end{equation}
The above functors are defined as follows:
\[
\left\{
  \begin{array}{lll}
    l(X)=(0\to 0\to \cdots \to X)  & \hbox{} \\
           & \hbox{} \\
   e(X_1\to X_2\to \cdots \to X_n)=X_n  & \hbox{} \\
   & \hbox{} \\
   r(X)=(X\xrightarrow{1_{X}} X\xrightarrow{1_{X}} \cdots \xrightarrow{1_{X}} X)  & \hbox{} 
  \end{array}
\right.
\]
\smallskip

\[
\left\{
  \begin{array}{lll}
   q(X_1\to X_2\to \cdots \to X_n)=(X_1\to X_2\to \cdots \to X_{n-1})  & \hbox{} \\
           & \hbox{} \\
   i(X_1\to X_2\to \cdots \to X_{n-1})=(X_1\to X_2\to \cdots \to X_{n-1}\to 0) & \hbox{} \\
   & \hbox{} \\        
   p(X_1\xrightarrow{f_{1}}  \cdots \xrightarrow{f_{n-1}} X_n)=(\Ker{(f_{n-1}\cdots f_2 f_1)}\to\Ker{(f_{n-1}\cdots f_2)}\to  \cdots \to\Ker{f_{n-1}}) & \hbox{} 
  \end{array}
\right.
\]
We claim that recollement (\ref{eq:heights}) lifts to a triangulated recollement of homotopy categories (the precise statement is Theorem~\ref{thm:n-lifting} below).

Consider the 2-cell (the identity natural transformation of the functor $e$), 
\[
  \xymatrix@C=3pc{
\big(\mathsf{Mor}_{n}R\big)_{\mathsf{GP}} \ar[r]^-{e}  \ar[d]^-{\mathrm{id}} &  (\Mod{R})_{\mathsf{GP}} \ar[d]^-{\mathrm{id}} \\
 \big(\mathsf{Mor}_{n}R\big)_{\mathsf{GI}} \ar[r]^-{e} &  (\Mod{R})_{\mathsf{GI.}}
  }
\]
We claim that the top functor $e$ is right Quillen. Since $e$ is exact it suffices to argue that it preserves fibrant and trivially fibrant objects. The first assertion is trivial since projective model structures have all objects fibrant. Now, if $(X_1\to \cdots \to X_n)\in \Rep_Q(\mathcal{P}^{<\infty}(R))$ (i.e.\ all $X_i$ have finite projective dimension) then $e(X_1\to \cdots \to X_n)=X_n$ is trivial and therefore the functor $e$ is right Quillen. Similarly, one can show that the bottom functor $e$ is left Quillen. In addition, the vertical identity functors are (left) Quillen equivalences (this follows from instance from \cite[Thm.~4.8]{Dal_Estrada_Holm}).

We now check the conditions of Theorem~\ref{thm:right_lifting} on the right-lifting of the abelian model structure $(\mathsf{Mor}_{n}(R))_{\mathsf{GP}}$ along the functor $i$. Let us first check that the functor $q$ sends  cofibrations in $(\mathsf{Mor}_{n}(R))_{\mathsf{GP}}$ to monomorphisms in $\mathsf{Mor}_{n-1}(R)$. A cofibration in $(\mathsf{Mor}_{n}(R))_{\mathsf{GP}}$ is a morphism of sequences $(X_1\to \cdots \to X_n)\to (Y_1\to \cdots \to Y_n)$, where all morphisms $X_i\to Y_i$ are certain monomorphisms. Applying the functor $q$ to such a morphism of sequences will delete the last monomorphism $X_n\to Y_n$ and the resulting morphism of sequences will be a monomorphism now in $\mathsf{Mor}_{n-1}(R)$. It remains to show that the right acyclicity condition $\leftperp{\F_{\mathsf{Mor}_{n-1}(R)}}\subseteq \W_{\mathsf{Mor}_{n-1}(R)}$ holds, where $\F_{\mathsf{Mor}_{n-1}(R)}=i^{-1}(\F_{\mathsf{Mor}_{n}(R)})$, $\W_{\mathsf{Mor}_{n-1}(R)}=i^{-1}(\W_{\mathsf{Mor}_{n}(R)})$. The fibrant objects $\F_{\mathsf{Mor}_{n-1}(R)}$ is the whole category $\mathsf{Mor}_{n-1}(R)$ since $\F_{\mathsf{Mor}_{n}(R)}=\mathsf{Mor}_{n}(R)$. Thus, the left perpendicular class $\leftperp{\F_{\mathsf{Mor}_{n-1}(R)}}$ consists of the projective objects of $\mathsf{Mor}_{n-1}(R)$. From \cite[Theorem~A]{HJ} these are sequences $P:=(P_1\to \cdots \to P_n)$ where all $P_i$ are projective $R$-modules and the maps $\phi_i^P$ are split monomorphisms. In particular, $i(P)=(P_1\to \cdots \to P_n\to 0)$ which clearly lies in $\W_{\mathsf{Mor}_{n}(R)}$. Hence the right-lifted hereditary abelian model structure $\mathsf{Mor}_{n-1}(R)_{\pi}$ exists. In fact, it easy to see that in this case the class of trivial objects of this model structure is identified with all objects in $\mathsf{Mor}_{n-1}(R)$ that degreewise have modules in $\mathcal{P}^{<\infty}(R)$. Hence the homotopy category $\Ho(\mathsf{Mor}_{n-1}(R))_{\pi}$ is identified with $\Ho(\mathsf{Mor}_{n-1}(R))_{\mathsf{GP}}$.
Similarly, one can show that the conditions of Theorem~\ref{thm:left_lifting} about the left-lifting of $(\mathsf{Mor}_{n}(R))_{\mathsf{GI}}$ along $i$ are fulfilled, and that $\Ho(\mathsf{Mor}_{n-1}(R))_{\mathsf{\iota}}$ is identified with $\Ho(\mathsf{Mor}_{n-1}(R))_{\mathsf{GI}}$.

The discussion so far shows that the recollement (\ref{eq:heights}) satisfies Setup \ref{Setup}. We now check that the two conditions from Theorem~\ref{thm:main_lift} are satisfied. 

We claim that $i$ is a right (or equivalently, a left) derived embedding. For this we use Proposition~\ref{prop:fully_faithful_right_derived} (iii). Take an object 
$(X_1\to \cdots \to X_{n-1})\in \mathsf{Mor}_{n-1}(R)$. Then applying the functor $i$ we get the sequence $(X_1\to \cdots \to X_{n-1}\to 0)$. We consider 
a cofibrant replacement of $(X_1\to \cdots \to X_{n-1}\to 0)$ in the abelian model structure $\mathsf{Mor}_{n}(R)_{\mathsf{GP}}$; in particular, this means that we have the following commutative diagram where the columns are short exact sequences:
\[
\xymatrix{
Z_1 \ar[d]^{} \ar[r]^{ } & Z_2 \ar[d]^{} \ar[r]^{ } & \cdots \ar[r]^{} & Z_{n-1} \ar[d]^{} \ar[r] & Z_n\ar[d] \\
 P_1 \ar[d]^{} \ar[r]^{ } & P_2 \ar[d]^{} \ar[r]^{ } & \cdots \ar[r]^{} & P_{n-1} \ar[d]^{} \ar[r] & P_n\ar[d] \\
 X_1 \ar[r]^{ } & X_2 \ar[r]^{ } & \cdots \ar[r]^{} & X_{n-1} \ar[r] & 0, }
\] 
here the $P_i$ are Gorenstein projective and the modules $Z_i$ have finite projective dimension. Applying the functor $q$ to the above diagram,  basically deletes the last column on its right hand side. Since the sequence $(P_1\to \cdots \to P_{n-1})$ lies in $\Phi(\GProj{\mathsf{Mor}_{n-1}(R)})$, it follows that the resulting morphism $(P_1\to \cdots \to P_{n-1})\to (X_1\to \cdots \to X_{n-1})$ is a cofibrant replacemenet ($\GProj$-approximation) in $\mathsf{Mor}_{n-1}(R)$. We infer that the functor $i$ is a right derived embedding. 

It remains to check condition (ii) of Theorem~\ref{thm:main_lift}. Let $X:=(X_1\to \cdots \to X_{n-1}\to X_n)$ be a cofibrant object in $\mathsf{Mor}_{n}(R)_{\mathsf{GP}}$, such that $e(X)$ is a trivial object, i.e.\ the $R$-module $X_n$ has finite projective dimension. We can easily see that the unit map $\eta_{X}\colon X\rightarrow (i\circ q)X$, which is the following map, 
\[
\xymatrix{
 X_1 \ar[d]^{\mathrm{id}} \ar[r]^{ } & X_2 \ar[d]^{\mathrm{id}} \ar[r]^{ } & \cdots \ar[r]^{} & X_{n-1} \ar[d]^{\mathrm{id}} \ar[r] & X_n\ar[d] \\
 X_1 \ar[r]^{ } & X_2 \ar[r]^{ } & \cdots \ar[r]^{} & X_{n-1} \ar[r] & 0 }
\] 
is a weak equivalence in $\mathsf{Mor}_{n}(R)_{\mathsf{GP}}$ (it is an epimorphism with trivial kernel, so in particular, a weak equivalence). 
Hence Theorem~\ref{thm:main_lift} applies and we obtain:
\begin{theorem}
\label{thm:n-lifting}
The recollement of abelian categories (\ref{eq:heights}) lifts to the following recollement of triangulated categories:
\[
\xymatrix@C=4pc{
\Ho(\mathsf{Mor}_{n-1}(R))_{\pi/\iota}  \ar@<0.0ex>[r] |-{\, \textbf{R}i\cong\textbf{L}i\,}   &  
\Ho(\mathsf{Mor}_{n}(R))_{\mathsf{\tiny{GP/GI}}}   \ar@/^1.4pc/[l] ^-{\textbf{R}p}    \ar@/_1.4pc/[l] _-{\textbf{L}q}  \ar[r] |-{\, \textbf{R}e\cong\textbf{L}e \,}&
\Ho(\Mod{R})_{\mathsf{GP/GI}} \ar@/_1.4pc/[l]  _-{\textbf{L}l}  \ar@/^1.4pc/[l] ^-{\textbf{R}r}
  }
\]
where $\Ho(\mathsf{Mor}_{n-1}(R))_{\pi/\iota}$ is the common homotopy category of the right and left transferred abelian model structures along the functor $i$. In fact, the homotopy category $\Ho(\mathsf{Mor}_{n-1}(R))_{\pi/\iota}$ coincides with $\Ho(\mathsf{Mor}_{n-1}(R))_{\mathsf{\tiny{GP/GI}}}$.
\end{theorem}

\subsection{Derived categories of Grothendieck categories}

We recall the following notion from \cite{Psa_homological_theory}.

\begin{definition}
\label{defnhomembed}
An exact functor $i\colon\G\to \G'$ between abelian categories is called a $k$-{\bf homological embedding}, where $k\geq 0$, if the map 
\[
\xymatrix{
i_{X,Y}^n\colon{\Ext}_{\G}^n(X,Y) \ar[r]^{ } & {\Ext}_{\G'}^n(i(X),i(Y)) }
\]
is invertible for all $X$, $Y$ in $\G$ and  $0\leq n\leq k$. The functor $i$ is called a {\bf homological embedding}, if $i$ is a $k$-homological embedding for all $k \geq 0$. 
\end{definition}

Let $\mathsf{R_{\mathsf{ab}}}(\A,\B,\C)$ be a recollement of Grothendieck categories with enough projective objects and consider the associated recollement of categories of complexes $\mathsf{R_{\mathsf{ab}}}(\Ch(\A),\Ch(\B),\Ch(\C)),$ as in Remark \ref{rem:induced_rec_cpx}. 
In what follows, we make use of the standard projective and injective model structures on chain complexes that were recalled before Remark \ref{rem:induced_rec_cpx}.

It is easy to see (cf.~the discussion in the introductory section  \ref{sec:intro}) that we have a Quillen adjoint triple 
\begin{equation}
\label{eq:Quillen_adj_tr_ch}
Q_{\mathrm{tr}}(\Ch(\B)_{\mathsf{proj/inj}}, \Ch(\C)_{\mathsf{proj/inj}}, \mathrm{id})
\end{equation}
as well as a Quillen adjoint triple
\begin{equation}
\label{eq:Quillen_adj_tr_ch2}
Q_{\mathrm{tr}}(\Ch(\A)_{\mathsf{proj/inj}}, \Ch(\B)_{\mathsf{proj/inj}}, \mathrm{id})
\end{equation}
where for both triples the morphisms $f,f',g,'g$ from Definition \ref{def:Quillen_adjoint_triple} are identity morphisms. In particular, it makes sense to ask when $i\colon \Ch(\A)\to \Ch(\B)$ is a derived embedding (as in Definition~\ref{defn:derembedding}).

\begin{proposition}\textnormal{(\cite[Theorem~6.9]{PsaroudVitoria})}
\label{prop:hom_embed}
The following are equivalent:
\begin{itemize}
\item[(i)] The functor $i\colon \Ch(\A)\to \Ch(\B)$ is a derived embedding.  
\item[(ii)] The functor $i\colon \A\to \B$ is a homological embedding.
\end{itemize}
\end{proposition}

\begin{proof}
We just need to check that the assumptions of \cite[Theorem~6.9]{PsaroudVitoria} are satisfied. Indeed, since we are working with Grothendieck categories with enough projectives, we have in particular that $\Der(\B)$ is TR5 and TR5$^{*}$ and that $\Der(\A)$ (as well as $\Der(\B)$) is left-complete, see for instance \cite[Example~6.4]{PsaroudVitoria}. Also the derived functor of $i$ preserves products and coproducts since there is an adjoint triple as in Theorem~\ref{prop:special_triple}. Hence \cite[Theorem~6.9]{PsaroudVitoria} applies.
\end{proof}

\begin{theorem}
\label{modelthmCPS}
Let $\mathsf{R_{\mathsf{ab}}}(\A,\B,\C)$ be a recollement of Grothendieck categories with enough projective objects. Then there exists a diagram of derived categories
\begin{equation}
\label{eq:der_dia}
  \xymatrix@C=4pc{
\Der(\A)   \ar@<0.0ex>[r] |-{\, \textbf{R}i\cong\textbf{L}i\,}   &  \Der(\B)  \ar@/^1.4pc/[l] ^-{\textbf{R}p}    \ar@/_1.4pc/[l] _-{\textbf{L}q}  \ar[r] |-{\, \textbf{R}e\cong\textbf{L}e \,}&
  \Der(\C)  \ar@/_1.4pc/[l]  _-{\textbf{L}l}  \ar@/^1.4pc/[l] ^-{\textbf{R}r}
  }
\end{equation}
which is a recollement if and only if the functor $i\colon \Ch(\A)\to \Ch(\B)$ is a derived embedding.
\end{theorem}

\begin{proof}
From the above remarks on the existence of the Quillen adjoint triples (\ref{eq:Quillen_adj_tr_ch}, \ref{eq:Quillen_adj_tr_ch2}), together with Propositions \ref{prop:transfer_trivial_proj} and \ref{prop:transfer_trivial_inj} we deduce that the recollement $\mathsf{R_{\mathsf{ab}}}(\Ch(\A),\Ch(\B),\Ch(\C))$ satisfies Setup~\ref{Setup}, hence from Theorem \ref{thm:main_lift} there exists a diagram as in (\ref{eq:der_dia}), which is a recollement if and only if conditions (i) and (ii) of Theorem \ref{thm:main_lift} hold. Hence, the proof will be over once we argue that condition (ii) of Theorem \ref{thm:main_lift} holds in the current situation. 
From the proof of Theorem \ref{thm:main_lift}, this boils down to showing that the canonically induced functor $j'\colon \Der(\A)\rightarrow\mathrm{Ker}(\mathbf{R}e)$  is essentially surjective. But this follows from \cite[Theorem~6.9]{PsaroudVitoria} (the fact that we can apply this last result in our context was explained in the proof of Proposition~\ref{prop:hom_embed}).
\end{proof}

We consider a ring $R$ with an idempotent element $e$. Then we have the following  recollement of module categories$\colon$
\[
\xymatrix@C=0.4cm{
\Mod{R/ReR} \ar[rrr]^{i:=f_*} &&& \Mod{R} \ar[rrr]^{e \ } \ar
@/_1.5pc/[lll]_{q:=R/ReR\otimes_R-}  \ar
 @/^1.5pc/[lll]^{p:=\Hom_R(R/ReR,-)} &&& \Mod{eRe}
\ar @/_1.5pc/[lll]_{l:=Re\otimes_{eRe}-} \ar
 @/^1.5pc/[lll]^{r:=\Hom_{eRe}(eR,-)}
 } 
\] 
where $f_*\colon \Mod{R/ReR}\to \Mod{R}$ is the restriction functor induced by the canonical morphism $f\colon R\to R/ReR$. Assume that $f_*$ is a homological embedding, equivalently, the morphism $f$ is a homological ring epimorphism \cite{GeigleLenzing}, equivalently, the ideal $ReR$ is stratifying \cite{CPS_3}. In this case, by Cline-Parshall-Scott \cite{CPS_3} it is known that the above recollement of module categories can be derived, that is, it induces a recollement $\mathsf{R_{\mathsf{tr}}}(\Der(R/ReR), \Der(R), \Der(eRe))$ of derived categories where all six functors are the derived functors of the underlying functors. Such a recollement situation was studied in \cite{PsaroudVitoria} under the name of stratifying recollement, see also \cite{AKLY_4}.
Theorem \ref{modelthmCPS} implies this well-known result, stated in homotopical terms.

\begin{corollary}
\label{cor:CPS}
Let $R$ be a ring with an idempotent element $e$. The following statements are equivalent$\colon$
\begin{enumerate}
\item The functor $i\colon \Ch(R/ReR)\to \Ch(R)$ is a derived embedding. 
\item There is a recollement of derived categories:
\[
 \xymatrix@C=4pc{
\Der(R/ReR)  \ar@<0.0ex>[r] |-{\, \textbf{R}i\cong\textbf{L}i\,}   &  
 \Der(R)  \ar@/^1.4pc/[l] ^-{\textbf{R}p}    \ar@/_1.4pc/[l] _-{\textbf{L}q}  \ar[r] |-{\, \textbf{R}e\cong\textbf{L}e \,}&
\Der(eRe). \ar@/_1.4pc/[l]  _-{\textbf{L}l}  \ar@/^1.4pc/[l] ^-{\textbf{R}r} }
\]
\end{enumerate}
\end{corollary}

\end{document}